\documentclass{article}
\usepackage{amsfonts,amsmath, amssymb,latexsym}%,mathrsfs,array}
\usepackage{multirow}
\def\Z{{\mathbb Z}}
\def\A{{\mathbb A}}

\def\SL{{\rm SL}}

\def\smax{{\rm smax}}
\def\GL{{\rm GL}}

\def\ram{{\rm ram}}
\def\Stab{{\rm Stab}}

\def\Sym{{\rm Sym}}
\def\Gal{{\rm Gal}}
\def\Cl{{\rm Cl}}

\newcommand{\cO}{{\mathcal O}}

\def\CS{{\mathcal S}}
\def\P{{\mathbb P}}

\def\Disc{{\rm Disc}}

\def\Aut{{\rm Aut}}

\def\irr{{\rm irr}}

\def\Frob{{\rm Frob}}

\def\Vol{{\rm Vol}}
\def\R{{\mathbb R}}
\def\F{{\mathbb F}}

\def\T{{\mathcal T}}

\newcommand{\cL}{\mathcal{L}}

\def\Ind{{\rm Ind}}

\def\FF{{\mathfrak F}}

\def\Q{{\mathbb Q}}

\def\U{{\mathcal U}}

\def\Z{{\mathbb Z}}
\def\G{{\mathbb G}}
\def\P{{\mathbb P}}
\def\F{{\mathbb F}}
\def\Q{{\mathbb Q}}
\def\C{{\mathbb C}}

\def\Sp{{\rm Sp}}

\def\fz1{{F_{\Z,1}}}

\def\max{{\rm max}}
\def\unr{{\rm unr}}
\def\nmax{{\rm nmax}}
\DeclareMathOperator{\nic}{\operatorname{n-ic}}
\def\univ{{\rm univ}}
\def\TT{\mathbb{T}}

\newcommand{\et}{\mathrm{et}}

\newtheorem{theorem}{Theorem}[section]
\newtheorem{corollary}[theorem]{Corollary}

\newtheorem{lemma}[theorem]{Lemma}
\newtheorem{example}[theorem]{Example}
\newtheorem{remark}[theorem]{Remark}
\newtheorem{proposition}[theorem]{Proposition}

\newenvironment{proof}{\noindent {\bf Proof:}}{$\Box$ \vspace{2 ex}}

\usepackage{xcolor,color,url,lmodern}

\title{Sato-Tate equidistribution of certain families of Artin $L$-functions}

\author{Arul Shankar, Anders S{\"o}dergren and Nicolas Templier}

\begin{document}

\maketitle

\begin{abstract}
We study various families of Artin $L$-functions attached to geometric
parametrizations of number fields. In each case we find the Sato-Tate
measure of the family and determine the symmetry type of the
distribution of the low-lying zeros.
\end{abstract}

\section{Introduction}

The Katz-Sarnak heuristics~\cite{KatzSarnak2} concern the arithmetic statistics
of a family $\FF$ of $L$-functions. In this paper,
we verify the heuristics for certain families arising from number
fields.  We shall follow the framework of the recent~\cite{SST}.  We
recall that in~\cite{SST} the authors distinguish two ways of forming
a family: harmonic families, which can be studied with the trace
formula; and geometric families arising from algebraic varieties
defined over the rationals.  In this paper we are concerned with the
geometric families of zero dimensional varieties, which give rise to
number fields.

The first family we study comes from the space $V$ of
monic polynomials of degree $n$.
To any $f\in V$ we
associate its scheme $X_f$ of zeros. This defines an affine subset
$X\subset V \times \A^1$. If $f\in V(\Z)$, then the ring $R_f:=\Z[T]/f(T)$ of
regular functions on $X_f$ is ${\it monogenic}$, which means that it
is generated by a single element called
a {\it monogenizer} of $R_f$.
The additive group $\G_a$ naturally acts on $V$ and on $\A^1$ via translations $(m\cdot f)(T):=f(T+m)$ and the covering $X\to V$ is $\G_a$-equivariant.

The ramification locus of the $n$-covering $X\to V$ is given by the
equation $\Delta=0$, where the discriminant $\Delta$ is a $\G_a$-invariant polynomial function on $V$.  The covering is \'etale away from the
ramification locus, thus in particular the ring $R_f$ is reduced if
and only if $\Delta(f)\neq 0$. The Galois group of the covering is the full permutation group $S_n$, which can be proved by identifying $V$ with the GIT quotient $\A^n // S_n$ and similarly $X\simeq \A^n // S_{n-1}$ with the natural projections $X\to V$ and $X\to \A^1$, see~\S\ref{s:n-ary}.

If $f\in V(\Z)^{\irr}$ is irreducible with nonzero discriminant, then the field of fractions $K_f$ of $R_f$ is a
number field of degree $n$. Let $M_f$ denote the normal closure of $K_f$.
The Galois group $\Gal(M_f/\Q)$ embeds irreducibly into $S_n$.  By composing
with the standard representation of $S_n$, we obtain an Artin
representation
\[
\rho_{K_f}:\Gal(M_f/ \Q)\hookrightarrow S_n \to\GL_{n-1}(\C).
\]
We are interested in the $L$-functions
$L(s,\rho_{K_f})$. Note that $\zeta(s)L(s,\rho_{K_f})$ is equal to the Dedekind zeta function $\zeta_{K_f}(s)$.
In a precise sense to be explained in Section~\ref{s:n-ary} below, for a 100\% of elements $f\in V(\Z)$, the polynomial $f$ is irreducible with nonzero discriminant and the normal closure $M_f$ has Galois group $S_n$. This can be seen to follow from an application of Hilbert irreducibility.

We consider the  subset $V(\Z)^\max$ of $V(\Z)^{\irr}$ consisting of irreducible polynomials $f$ with nonzero discriminant such that $R_f$ is a maximal order in $K_f$.
Imposing the condition of maximality requires the application of a sieve and a tail estimate developed and proved in \cite{BSW}.
The action of $\G_a(\Z)=\Z$ by translation preserves the subsets $V(\Z)^{\irr}$ and $V(\Z)^{\max}$ of $V(\Z)$. 
Let the family $\FF$ consist of the $\Z$-orbits on $V(\Z)^{\max}$. 
For a $100\%$ of $f\in V(\Z)^\max$, the representation $\rho_{K_f}$ has image $S_n$, hence $L(s,\rho_{K_f})$
is cuspidal and orthogonal self-dual.

% he family is \emph{essentially cuspidal} in the sense of~\cite{SST}.
% It is \emph{homogeneous orthogonal}~\cite{SST} because the dot product on $\C^{n-1}$ is a symmetric bilinear form preserved by~$S_n$.

The family $\FF$ parametrizes monogenized number fields of
degree $n$ over $\Q$ up to isomorphism.  If $R=\Z[\alpha]$ is a monogenic ring, then the pair $(R,\alpha)$ is called a {\it monogenized ring}.  A pair
$(K,\alpha)$ where $K$ is a number field is said to be a {\it
  monogenized field} if $\alpha$ belongs to $\cO_K$, the ring of
integers of $K$, and the pair $(\cO_K,\alpha)$ is a monogenized
ring. Two monogenized fields $(K,\alpha)$ and $(K',\alpha')$ are said
to be {\it isomorphic} if $K$ is isomorphic to $K'$ and this
isomorphism carries $\alpha$ to $\alpha'+m$ for some integer $m\in\Z$.
If a monic polynomial $f$ is irreducible, then the field of
fractions of $R_f$ is a degree-$n$ field $K_f=\Q[T]/f(T)$, and the pair
$(K_f,\alpha)$ is a monogenized field, where $\alpha$ is the image of
$T$ in $R_f$.  Conversely, if $(K,\alpha)$ is a monogenized field,
then the characteristic polynomial of $\alpha$ is an element $f$
belonging to $V(\Z)^{\max}$, and the field of fractions of $R_f$ is
$K$.

It is possible for number fields to have more than one monogenizer. However, a result of Birch and Merriman \cite{BM} implies that a number field has only finitely many monogenizers, up to translation by a rational integer. Therefore, a number field $K$ arises only finitely many times in the family $\FF$. 

Since $V\simeq \A^n // S_n$ it is natural to consider the associated grading. More precisely, an element $(x_1,\ldots,x_n)\in\A^n//S_n$ gives rise to the polynomial $f(T)=\prod_i(T-x_i)$. Considering the $x_i$ to be elements of degree 1, if follows that for $f(T)=T^n+a_1T^{n-1}+\cdots +a_n\in V$, the coefficient $a_i$ has degree $i$ because it is $(-1)^i$ times the $i$-th symmetric polynomial evaluated at the roots of $f$. The discriminant $\Delta$ is then homogeneous of degree $n(n-1)$. We order the family by the height $h(f)=\max_i \{|a_i|^{n(n-1)/i}\}$ on $V(\R)$ which is also homogeneous of degree $n(n-1)$. We then prove the following theorem (see Sections \ref{s:setup} and \ref{s:n-ary}):

\begin{theorem}\label{th:nic}
The family parametrizing monogenized degree-$n$ number fields ordered
by height has Sato-Tate group $S_n \subset \GL_{n-1}(\C)$, and thus
Symplectic symmetry type.
\end{theorem}

The first assertion of Theorem~\ref{th:nic} is the Sato-Tate
equidistribution for families in the sense of~\cite[Conj.~1]{SST}. If
we let $\T_n$ be the set of conjugacy classes in $S_n$, then this
means that, as $x,y\to \infty$ with $\frac{\log x}{\log y}$ large
enough, the elements
\begin{equation}\label{i:ST}
\{\rho_K(\Frob_p):\ K\in \FF(x),\ p<y\} \subset \T_n
\end{equation}
become equidistributed for the Sato-Tate measure on $\T_n$ which is
the pushforward of the normalized counting measure on $S_n$. 
%The first assertion on the symmetry type follows because the representation $S_n\to\GL_{n-1}(\C)$ is orthogonal, see~\cite[Conj.~2]{SST}. 
This is to be compared with the Chebotarev equidistribution theorem
that says that for any $S_n$-number field $K$, the elements
$\{\rho_K(\Frob_p):\ p<y\}$ are equidistributed in $\T_n$ as $y\to
\infty$. Here the extra averaging over $K\in \FF(x)$ allows us to
produce a quantitative power saving error term.

In general, for any given family the Sato-Tate equidistribution~\eqref{i:ST} has applications to sieving, zero density results, averaging of $L$-values, and low-lying zeros.
In this paper we confine ourselves to the latter aspect.
The second assertion of Theorem~\ref{th:nic} on the Symplectic
symmetry type corresponds to the one-level density with restricted
support of low-lying zeros. As
explained in~\cite[Conj.~2]{SST} the proof shall proceed from the Sato-Tate
equidistribution of~\eqref{i:ST} and from considering the following
two additional quantities. First, the \emph{root numbers} of
$L(s,\rho_{K})$ are always $+1$ because the root numbers of both
$\zeta(s)$ and $\zeta_K(s)$ are $+1$. This also follows from $\rho_K$
being an orthogonal representation as a special case of a result of
Fr\"ohlich-Queyrut \cite{FQ}. Second, the \emph{rank of the family} is zero, see
\S\ref{sub:averageFrob}.

%* below paragraph, check max versus sf, and where irr is needed* 

 The proof of Theorem \ref{th:nic} and of the equidistribution of
\eqref{i:ST} proceeds as follows: first, we determine asymptotics for
the number of $\Z$-orbits on $V(\Z)^\max$ having bounded height, and
whose coefficients satisfy any finite set of congruence conditions. It
is here that we need the sieve methods of \cite{BSW}. Next, note that
$\rho_{K_f}(\Frob_p)$ is determined by $R_f\otimes\F_p$, which is the
ring over $\F_p$ corresponding to the reduction of $f$ modulo $p$. We
then determine the density of elements in $V(\Z)^\max$, such that the
corresponding value of $\rho_{K_f}(\Frob_p)$ is fixed, via a local
count of configurations of $n$ points in $\F_p$.

Next, it is desirable to have families that count each number field at
most once. This is achieved in the cubic case by further considering
orbits under the $\GL_2$ action. We refine the construction by forming
the affine space $V\simeq \A^4$ of binary cubic forms, and construct a
$3$-covering $X\to V$ as above, except that now $X\subset V\times
\P^1$ is quasi-projective.  We consider the action by $\GL_2$ on $V$
and on $\P^1$ which induces an equivariant structure of the covering
$X\to V$, i.e., the action of $\GL_2$ on $V\times\P^1$ preserves $X$
and thus the map $X\to V$ is compatible with the actions of $\GL_2$ on
$X$ and $V$. In fact $V$ is a prehomogeneous vector space for this
action and the discriminant $\Delta$ is a generator of the algebra of
invariant polynomials. We then consider elements $f$ in
$\GL_2(\Z)\backslash V(\Z)^{\smax}$ as parameters for maximal
$S_3$-orders. Since two maximal cubic forms give rise to the same
cubic field $K_f$ if and only if they belong to the same
$\GL_2(\Z)$-orbit, we obtain a family $\FF$ which parametrizes the
$S_3$-fields exactly once. We shall order the family by discriminant
so that $\FF(x)$ coincides with the set of $S_3$-fields with absolute
discriminant less than $x$. It is a result of Davenport--Heilbronn
that $|\FF(x)| \sim x/(3\zeta(3))$ as $x\to
\infty$. Bhargava~\cite{B1,B2} proved the analogous result for quartic
and quintic fields. The following is due to A.\ Yang~\cite{Yang} in
the cubic and quartic cases.

\begin{theorem}\label{thmainsn}
The families parametrizing $S_3$-, $S_4$- and $S_5$-fields ordered by discriminant are homogeneous orthogonal, and thus have Symplectic symmetry type.
\end{theorem}

The thesis~\cite{Yang} is unpublished.  An account first appeared
in~\cite{SST} and Sections \ref{s:setup} and \ref{s:families} of this
paper provide more details.  A different treatment is given
in~\cite{ChoKim1,ChoKim2}.  The advantage of our treatment compared
to~\cite{Yang,ChoKim1,ChoKim2} is to make transparent the relation
between the symmetry type and the other statistical invariants of the
families.  As before the statement is to be interpreted in the sense
of the quantitative equidistribution of~\eqref{i:ST}, where the
measure on $\T_n$ is the pushforward of the normalized counting
measure on $V(\Z_p)^\max$.  We recall the concept of an homogeneous
orthogonal family in \S\ref{s:setup}.  A key aspect of the proofs of
both Theorem~\ref{th:nic} and Theorem~\ref{thmainsn} is the study of
maps
\begin{equation}\label{local-density}
V(\Z_p)^\max  \supset V(\Z_p)^\unr \twoheadrightarrow
	V(\F_p)^{\Delta\neq 0} \twoheadrightarrow \T_n,
\end{equation}
which gives the splitting type of an order unramified at $p$ in terms of
the reduction of the corresponding polynomial modulo $p$.

We shall rely in an essential way on Bhargava's work on counting and parametrizing quartic and quintic fields.  
A rank $n$ ring arises as the ring of functions of a projective set of
$n$ points defined over $\Z$, and conversely its spectrum is a set of
$n$ points. As explained in \cite[\S2]{Bquintic}, every rank $n$ ring
arises from a set of $n$ points in $\P^{n-2}$. In the case $n=3$,
where $3$ points in $\P^1$ are parametrized as the zero set of binary
cubic forms, the above construction was sufficient.
Binary $n$-ic forms parametrize sets of $n$ points in $\P^1$ which,
for $n\geq 4$, do not give rise to all rank $n$ rings,
see~\cite{Woodbnf}. To parametrize all rank $n$ rings
for $n=4,5$, Bhargava writes the $n$ points in $\P^{n-2}$ as the
intersection of quadrics. For $n=4$, a generic set of two quadrics in
$\P^2$ intersect in $4$ points. Furthermore, every set of $4$ points
in $\P^2$ arise this way. Thus quartic rings are naturally
parametrized by pairs of ternary quadratic forms~\cite{Bquartic}. We
denote the underlying space $V=2\otimes \Sym^2(3)$.
 
 In the case $n=5$, five quadrics are required to obtain an
 intersection of $5$ points. However a generic set of five quadrics do
 not intersect at all in $\P^3$. Rather it is known from the work of
 Buchsbaum and Eisenbud \cite{BE} that five quadrics in $\P^3$
 intersect in five points if and only if they arise as the $4\times 4$
 Pfaffians of an alternating $5\times 5$ matrix of linear forms in
 four variables. The underlying space is $V=4\otimes \wedge^2(5)$
 which give rise to a parametrization of quintic
 rings~\cite{Bquintic}.

In both cases $n=4,5$ we obtain a quasiprojective scheme $X\subset
V\times \P^{n-2}$ cut out by quadrics. This is a branched covering
$X\to V$ of degree $n$. As in the case $n=3$, the covering has an
equivariant $G$-structure with $G=\GL_2\times \SL_3$ if $n=4$ and
$G=\GL_4\times \SL_5$ if $n=5$. (Here, $\GL_2\times\SL_3$ acts on
$\P^2$ via the action of $\SL_3$ and $\GL_4\times \SL_5$ acts on
$\P^3$ via the action of $\GL_4$.) As before we let
$V(\Z)^{\mathrm{max}}$ (resp.  $V(\Z)^{\mathrm{smax}}$) be the set of
forms that give rise to maximal rings (resp. maximal $S_n$-rings).  As
before, we consider elements $f$ in $G(\Z)\backslash V(\Z)^{\rm
  smax}$ as parameters for maximal $S_n$-rings. These are the families
$\FF$ studied by Bhargava which parametrize $S_4$- and $S_5$-fields.
The sets $\FF(x)$ will be ordered by discriminant and the asymptotics
$|\FF(x)|\sim c x$ as $x\to \infty$ are the celebrated results
of~\cite{B1,B2}. Compared to the counting in Theorem~\ref{th:nic}
ordered by height, a major difficulty for these families ordered by
discriminant, overcome by Bhargava, is the presence of non-compact
``cusps", which means there are forms in a fundamental domain for the
$G(\Z)$-action on $V(\Z)$ that have large coefficients but small
discriminant. The Sato-Tate equidistribution~\eqref{i:ST} with a power
saving error term is obtained in~\cite{BBP,B1,B2,ShTs}.

For $n=4,5$, restricting to the nonsingular locus gives \'etale
coverings $X^{\Delta\neq 0}\to V^{\Delta\neq 0}$. Quotienting
$G$ by the subgroup that acts trivially on $X$, we obtain
an algebraic group $H$ such that $H(\C)$ acts transitively on
$V^{\Delta\neq 0}(\C)$ and acts simply transitively on
$X^{\Delta\neq 0}(\C)$. 
(See \cite[Table 1]{BSWPre} for an exact description of $H$.)
The stabilizer in $H(\C)$ of any element
in $V^{\Delta\neq 0}(\C)$ is known to be $S_n$, see \cite[\S
  7]{SatoKimura} for $n=4$ and \cite[Proposition 2.13]{Wright-Yukie}
for $n=5$. It then
follows that the normal closure of the \'etale covering
$X^{\Delta\neq 0}(\C)\to V^{\Delta\neq 0}(\C)$ has Galois group
$S_n$.  A corollary of the equidistribution~\eqref{i:ST} yields an
arithmetic proof of this algebraic result. In fact, the entire
equidistribution is not necessary; surjectivity onto $\T_n$ would
suffice. 
For the initial
family of monic degree-$n$ polynomials with $n\ge 2$, the subset
$V(\C)^{\Delta\neq 0}$ of polynomials with non-zero discriminant admits again an \'etale covering $X^{\Delta\neq 0}(\C)$ defined by their zero
locus in $\C$. The normal closure of this covering has Galois group $S_n$, a result which is well-known.
Conversely the monodromy group being the full $S_n$ is
closely related to the equidistribution~\eqref{i:ST}, and
Theorems~\ref{th:nic} and~\ref{thmainsn}.  This is a special case
of~\cite{Katz} and~\cite[\S2.11]{SST} in the case of geometric
families of higher dimensional varieties.

% todo
% Construction of covering groups with given monodromy is an interesting problem in arithmetic and algebraic geometry. Among many results we mention the recent ..
% I'll add references to Bary Soroker and Entin who provide several difficult examples of families with full S_n monodromy group. Also Yun. And say that for any of these families the same techniques as in the present paper should apply. I'll reread \S2.2 and 3.4 of SST and perhaps add a comment regarding Noether conjecture.

As a side remark it is interesting to note that non-isomorphic $S_n$-number fields $K_f$ and $K_{f'}$ have distinct Dedekind zeta functions (see~\cite{Perlis}). Since each $S_n$-field occurs exactly once, we are counting the $L$-functions $L(s,\rho_{K_f})$ also with multiplicity one. 

It is believed that for any $S_n$-number field $K$, the central value
$\zeta_K(\frac12)$ is nonzero. This belief is reinforced by the
Symplectic symmetry type of the families described above, which
thereby exhibit a repulsion of the low-lying zeros at the central
point. For quadratic fields the non-negativity of $\zeta_K(\frac12)$
implies (see \cite{Iwaniec}) a strong effective lower bound on the
class number of $K$. For $S_5$-number fields the non-vanishing of
$\zeta_K$ for real $s\in(0,1)$ is a useful hypothesis in establishing
modularity in~\cite{Calegari}.

Unconditionally Soundararajan~\cite{Sound} has proved that a positive
proportion of all quadratic number fields satisfy
$\zeta_K(\frac12)\neq 0$, which is also strengthened
in~\cite{ConreySound} into a positive proportion of non-vanishing of
$\zeta_K(s)$ for $s$ real between zero and one. The generalization to
families of $S_n$-number fields with $n\ge 3$ is still open. Our
Theorems~\ref{th:nic} and \ref{thmainsn} above are not yet strong
enough to derive a result in this direction because of the restricted
support of the one-level density.

The Sato-Tate equidistribution in $\T_n$ for the above families is
related to mass formulas~\cite{B3}. The families are
homogeneous orthogonal because the Frobenius--Schur indicator of
$S_n\subset \GL_{n-1}(\C)$ is equal to $+1$, an observation which was
also made in~\cite[Item 76]{Kowalski}.  Another interesting application to an
analogue of the Erd\"os--Kac theorem appears in~\cite{Oliver-Thorne},
and to average upper-bounds for class numbers in~\cite{EPW}. Although
not stated in~\cite{EPW}, it can be verified that their sieving
argument applies to any number fields family ordered by discriminant
that satisfies the Sato-Tate equidistribution in the sense
of~\cite[Conj.~1]{SST}.

As stated above, the root number is $+1$ for any $S_n$-number
field. In general the root number of a self-dual Artin representation
may be $-1$, the first example was given by
Armitage~\cite{Armitage}. Thus one may wonder what happens for general
families of Artin representations with a different root number. This
motivates our study of families of quaternionic fields. Let $K$ be a
quaternionic field, that is a degree eight number field whose Galois
group $\Gal(K / \Q)$ is the quaternion group $Q$ of eight
elements. There is a unique irreducible two-dimensional representation
of $Q$ and we can attach an Artin representation
\[
\rho_{K}:\Gal(K/ \Q)\simeq  Q \to\GL_{2}(\C)
\]
which is symplectic. We can view $\rho_K$ as induced from a Hecke
character of order $4$ in a quadratic extension of $\Q$ inside
$K$. Furthermore, it is known to correspond to an automorphic form on
$\mathrm{PGL}_{2}$, precisely to a (dihedral) Maass form of weight
$0$, eigenvalue $\frac14$ and trivial nebentypus, see \cite[\S 3]{BS} and the references therein.

\begin{example}
	(i) Dedekind found that $K=\Q\Bigl(\sqrt{(2+\sqrt{2})(3+\sqrt{6})}\Bigr)$ is a quaternionic extension of $\Q$ containing $\Q(\sqrt{2},\sqrt{3})$ (see \cite{Dedekind}). The root number of $\rho_K$ is $+1$.

(ii) The field $K=\Q\Bigl(\sqrt{(5+\sqrt{5})(41+6\sqrt{41})}\Bigr)$ is a quaternionic extension of $\Q$ containing $\Q(\sqrt{5},\sqrt{41})$. The root number of $\rho_K$ is $-1$.
\end{example}

To form a family, we fix an arbitrary quaternionic field $K$. Let $q\equiv 0,1 \pmod{4}$ be a fundamental discriminant that is coprime with the discriminant of $K$. Let $\chi_q$ be the associated quadratic Dirichlet character which we may also view as an Artin representation onto $\{\pm 1\}\subset \GL_1(\C)$. Consider the Artin representation that is the character twist $\rho_K\otimes \chi_q$. Since $\Gal(K(\sqrt{q})/\Q) \simeq Q \times \Z/ 2\Z$ and the representation factors through the unique non-trivial surjection $Q\times \Z/ 2\Z \to Q$, which defines a unique quaternionic field $K_q$, $\rho_K\otimes \chi_q$ is the same as $\rho_{K_q}$. We call the field $K_q$ a quadratic twist of $K$ and obtain in this way a one-parameter family of quaternionic fields.  

In Section~\ref{s:quaternion} we give an equivalent description of $K_q$ using a theorem of Witt~\cite{Witt}, and relate this to a similar construction by Fr\"ohlich~\cite{Fr}. This description also shows that the family is geometric. The following is essentially due to Rubinstein~\cite{Rubinstein}.
\begin{theorem}\label{t:quaternion}
Let $K$ be a quaternionic field and consider the above one-parameter family of quaternionic Artin representations of $K_q$ parametrized by discriminants $q$. The family is homogeneous symplectic and it has $\mathrm{SO}(\mathrm{even})$ symmetry type if $\rho_K$ has root number $+1$ and $\mathrm{SO}(\mathrm{odd})$ symmetry type if $\rho_K$ has root number $-1$.  
\end{theorem}

In contrast to the quadratic twists of an elliptic curve where the root numbers fluctuate, we note the interesting phenomenon that the root numbers of the quadratic twists $K_q$ of a quaternionic field are constant. We verify this in Section~\ref{s:quaternion} where we give a brief exposition of the arithmetic of quaternionic fields gathering several results scattered in the literature.

Suppose that $\rho_K$ has root number $-1$. It is believed that $\text{ord}_{s=\frac12} L(s,\rho_{K}) = 1$ and similarly for $K_q$ for all $q$. This belief is reinforced by the  $\mathrm{SO}(\mathrm{odd})$ symmetry type of the family which defines the same determinantal point process as the union of $\Sp(\infty)$ and a single zero at $\frac12$.

In Section~\ref{s:other} we investigate two situations where one constructs a geometric family starting from another. The first construction is due to Davenport-Heilbronn. Starting from a binary cubic form $f \in V(\Z)$ we attach the quadratic field whose discriminant is $\Delta(f)$. Geometrically this is a branched covering of $V$ of degree two which is again $\GL_{2}$-equivariant (to be compared with the branched covering of degree three parametrizing cubic fields). It is famously used to determine asymptotically the average size of the $3$-part of the class group of quadratic fields. Unsurprisingly, we show in Section~\ref{s:other} the Sato-Tate equidistribution in $\T_2$ for this family. Similarly the second construction comes from Bhargava's parametrization of the pairs of quartic rings together with their resolvent rings. This yields a $\GL_{2}\times \SL_{3}$-equivariant covering of degree three which can be used to determine the average size of the $2$-part of the class group of cubic fields~\cite{B1}. We prove that the Sato-Tate equidistribution in $\T_3$ holds for this family.

In all of the above families the rank of the family is zero in the sense of~\cite{SST}. The average trace of Frobenius is a Weil number of integer weight which geometrically comes from the fact that we are counting orbits of points of varieties over finite fields.  Thus it is always the case that the rank is zero for any  geometric family of number fields because the construction involves the $H^0$ of the zero-dimensional fibers. This is consistent with the belief that Artin $L$-functions never vanish at the central point except when forced by the root number being $-1$.

It would be interesting to obtain similar results when $\FF(x)$ is the
set of all $S_n$-number fields of discriminant at most $x$.  It is
possible to view $\FF(x)$ as a parametric set by considering the
configuration of $n$ points in $\P^{n-2}$ modulo the action by
$\GL_{n-1}$. For $n\geq 4$, this yields an algebraic variety $V$ which
can always be cut out by a certain number of quadrics~\cite[Theorem
  138]{Wilson}. A conjecture of Bhargava~\cite{B3} predicts an
asymptotics $|\FF(x)| \sim c_n x$ as $x\to \infty$ and moreover the
mass conjecture~\cite{B3} would also imply the Sato-Tate
equidistribution in the same way as we have proceeded for the other
families of the present paper. For $n\le 5$ the variety $V$ can be
parametrized by a prehomogeneous group action on a vector space by the
results of Davenport-Heilbronn~\cite{DH} and
Bhargava~\cite{Bquartic,Bquintic} as mentioned above and $\FF$ becomes
a parametric family in the sense of~\cite{SST}. For $n\ge 6$ this is
not the case and thus the study of rational points in $V(\Z)$ is an
extremely delicate problem. For the same reason it is not possible to
include such parameter spaces in the definition of geometric families
in~\cite{SST}; working in such complete generality would allow too
many pathologies in the asymptotic of families, see~\cite[\S3.1]{SST}.

As explained above the families are obtained by a sieving process of the forms $f\in V(\Z)$. In this process we can extract the forms $f$ that give rise to number fields with a constant $S_n$ Galois group. It is interesting to study what happens if we form families starting from the same space but without sieving. Then the Galois group of $M_f/ \Q$ can vary with $f$. So we call these mixed families. These mixed families fit in the framework of~\cite{SST} and  we shall explain that the Sato-Tate equidistribution holds for them as anticipated in~\cite{SST}. Interestingly it is shown in~\cite{Wood:D4} that the family of $D_4$-fields ordered by discriminant does not have a mass formula. The Sato-Tate measure is a linear combination of Sato-Tate measures attached to the Haar measures on different Galois groups which occur with positive proportion. Serre also describes the possible Sato-Tate measures in this way in his recent book~\cite{Serre:NXp}. One interesting case is the mixture of $S_4$- and $D_4$-fields arising from pairs of ternary quadratic forms, see Section~\ref{s:mixed}.  Incidentally the quantitative equidistribution for the	family of $D_4$-fields is not yet established.

Let us also mention some other open questions that arise from our perspective on families and on which we hope to return elsewhere. Besides the one-parameter families explored in Section~\ref{s:quaternion} it would be interesting to study other families of quaternionic fields (see \cite{Kluners, FLPS}). In this paper we do not consider lower order terms as in, e.g., \cite{Anders}; these can be seen to be related to the counting measures on $V(\Z/p^r\Z)$ of Section~\ref{s:mass}. Finally, it should be possible to improve the remainder terms and support for one-level density using for example large sieve inequalities and Fourier transforms of orbital measures. 

% It is clear that the results that we establish in this paper are already present or anticipated in some form in the literature cited. Our contribution is to present this material in a new way, organized around equidistribution properties and local counting. One goal of the present paper is a detailed exposition of this perspective. 

% One goal of this paper was to investigate families of number fields in a systematic way, organized around Sato-Tate equidistribution. One goal of the present paper is a detailed exposition of this perspective. 
% possible centered on Sato-Tate equidistribution with potential applications in mind.

\subsection*{Acknowledgments} 
We thank Manjul Bhargava, Peter Sarnak and Jacob Tsimerman for many
helpful discussions and Melanie Wood for comments and for sending us~\cite{EPW}. We thank the referee for a careful reading and helpful suggestions. Most of this work was done while the authors were
at Princeton and it is a pleasure to thank the IAS for providing
excellent working conditions. We also enjoyed the hospitality of
Harvard University, Boston College, and Carleton University. 
% Results were announced by the third-named author at various talks in Harvard (April 2012), Jeju (August 2012), Lausanne (March 2013), ICERM (April 2013), Jerusalem (December 2013), Ottawa (June 2014), and by the second-named author in Turku (May 2014). After the ICERM talk P.\ Cho and H.\ Kim wrote two articles~\cite{ChoKim1,ChoKim2} concerning the families in Theorem~\ref{thmainsn}. 
% Although there is some overlap there (and with the thesis of A.Yang \cite{Yang}), our exposition turns out to be rather different. 
The first-named author was partially supported by NSF grant DMS-1128155. The second-named author was
supported by NSF grant DMS-1128155, and by a grant from the Danish Council for Independent Research
and FP7 Marie Curie Actions-COFUND (grant id: DFF-1325-00058). The
third-named author was supported by NSF grant DMS-1200684.

\section{General setup for zeta functions of degree $n$ number fields}\label{s:setup}
Let $K$ be a degree-$n$ number field with normal closure $M$. Then $\zeta_K(s)$ is the
Artin $L$-function corresponding to the trivial representation of the
absolute Galois group $\Gal(K)$. The Galois group $\Gal(K)$ is an
index $n$ subgroup of $\Gal(\Q)$ and we have
$$\zeta_K(s) = L(s,\Ind_{\Gal(K)}^{\Gal(\Q)}1).$$ The representation
$\Ind_{\Gal(K)}^{\Gal(\Q)}1$ of $\Gal(\Q)$ factors through
$\Gal(M/\Q)\hookrightarrow S_n$ and decomposes into the direct sum of
the trivial representation and the composition with the standard representation $\rho :
S_n\to\GL_{n-1}(\C)$.  Therefore, we
have 
\begin{equation}\label{Dedekind-relation}
\zeta_K(s)=\zeta(s)L(s,\rho_K),
\end{equation} 
where $\zeta(s)$ is the
Riemann zeta function and $L(s,\rho_K)$ is the Artin $L$-function
corresponding to
\begin{equation}\label{rhoK}
\rho_K:\Gal(\Q)\to\Gal(M/\Q)\hookrightarrow  S_n\to\GL_{n-1}(\C).
\end{equation}
Note that the conductor of $L(s,\rho_K)$ is equal to the conductor
of $\zeta_K(s)$. We denote it by $C_K$. It follows from Artin's
conductor-discriminant formula that $C_K$ is equal to the absolute
value of the discriminant $\Delta(K)$ of the number field $K$.
% Since the zeros of $\zeta(s)$ are discrete, for the purposes of studying low-lying zeros, it suffices to study the zeros of $L(s,\rho_K)$.

We will study the statistics of the low-lying zeros of $L(s,\rho_K)$
by summing these zeros against a test function always denoted by
$f$. We pick $f$ to be an even Paley-Wiener function on $\R$, in the sense that its Fourier transform
\begin{equation}\label{eqft}
\widehat{f}(x):=\int_{-\infty}^\infty f(y)e^{-2\pi ixy}dy
\end{equation}
is smooth and of compact support. If $\widehat{f}$ has
support contained in $[-\alpha,\alpha]$, then $f$ can be
extended to an entire function of exponential type $\alpha$.
The first step towards understanding the statistics of the zeros of
$L(s,\rho_K)$ is the explicit formula.  For each $m\ge 1$, we write
$\lambda_K(m)$ for the Dirichlet coefficients of $L(s,\rho_K)$. Note
that $\lambda_K$ is integer valued and that $\sum_{d\mid m} \lambda_K(d)$ is the number of ideals of $K$ of
norm $m$. We write the logarithmic derivative of $L(s,\rho_K)$ for
$\Re(s)>1$ as
\begin{equation}\label{logder}
-\frac{L'}{L}(s,\rho_K)=\sum_{m=1}^\infty \frac{\theta_K(m)\Lambda(m)}{m^{s}},
\end{equation}
where $\Lambda$ is the von Mangoldt function. We state the explicit
formula in the form of~\cite[Proposition 2.1]{RS}:

	We write the nontrivial zeros of $L(s,\rho_K)$ as
        $\tfrac12 +i\gamma^{(j)}_K$, where the imaginary parts of
        $\gamma^{(j)}_K$ have absolute value bounded by
        $1/2$. (Under GRH, the $\gamma^{(j)}_K$ are real.) Similarly we denote the poles of $L(s,\rho_K)$ by $\tfrac12 + ir_K^{(j)}$. 
%(Under the Artin conjecture there are no such poles.) 
\begin{proposition}\label{propexplicit}
With notation as above, if $K$ is a degree-$n$ number field and $f$ is an even Paley-Wiener function, then 
\begin{equation}\label{eqexplicit}
	\sum_j f\big(\gamma^{(j)}_K\big) 
- \sum_jf\big(r^{(j)}_K \big)
=\frac1{2\pi}\int_{-\infty}^\infty f(t)(\log
C_K+O(1))dt -\frac{1}{\pi}\sum_{m=1}^\infty
\frac{\theta_K(m)\Lambda(m)}{\sqrt{m}}\widehat{f}\Bigl(\frac{\log m}{2\pi}\Bigr).
\end{equation}
\end{proposition}

\subsection{Frobenius and splitting types} Let $p$ be a prime that is unramified in $K$. Let $\cO_K$ denote the ring of integers of $K$ and write
$$
\cO_K/(p)=\F_{p^{f_1}}\oplus\F_{p^{f_2}}\oplus\cdots\oplus\F_{p^{f_k}},
$$ with $f_1\geq f_2\geq\cdots\geq f_k$.  Then the {\it splitting
  type} of $p$ in $K$ is defined to be $(f_1f_2\ldots f_k)$. Thus, the
set of possible splitting types for unramified primes can be naturally
identified with the set of partitions of~$n$, or equivalently with
$\T_n$, the set of conjugacy classes of $S_n$.

Our goal now is to relate the splitting type $\tau\in \T_n$ of $p$ to
the coefficients of the Euler factor at $p$ of the $L$-function
$L(s,\rho_K)$.  To this end, we need to relate it to the Frobenius
conjugacy class of $p$ in $\Gal(M/\Q)$. We follow the short and
elegant exposition of Wood \cite{Wood}.

Let $\mathfrak{p}\subset M$ be a fixed prime ideal lying above the unramified prime $p$, and let $G_{\mathfrak{p}}$ denote the decomposition group.  This group is cyclic and is generated by $\Frob_{M/\Q}\mathfrak{p}$. The conjugacy class of $\Frob_{M/\Q}\mathfrak{p}$ is independent of the choice of $\mathfrak{p}$ above $p$ and from now on we denote this class by $\Frob_p$. 
Then the splitting type of $p$ and the action of $\Frob_p$ correspond
to the same partition of $n$. Equivalently $\rho(\tau)$ and
$\rho(\Frob_p)$ are conjugate. We denote this by writing $\rho(\tau)\sim\rho(\Frob_p)$.

\begin{lemma}\label{lemap}
	Let $\chi$ denote the character of the standard representation $\rho$ of $S_n$. If $p$ is unramified in $K$ and its splitting type is $\tau\in\T_n$, then we
have $\theta_K(p^k)=\chi(\tau^k)$ for all $k\ge 0$. Furthermore, for any rational prime $p$ and $k\ge 0$, we have $|\theta_K(p^k)|\leq n-1$.
\end{lemma}

\begin{proof}
Suppose that the splitting type of $p$ in $K$ is $\tau\in\T_n$.  Since
$p$ does not ramify in $K$, the Euler factor $L_p(s,\rho_K)$ at $p$ is equal to
\[
\det(I-p^{-s}\rho_K(\Frob_p))^{-1}=\det(I-p^{-s}\rho(\tau))^{-1}
=\prod_{i=1}^{n-1}(1-  \alpha_i p^{-s})^{-1}
%=\sum_{k=0}^\infty \mathrm{tr}(\rho(\tau)^k)p^{-ks}.
\] 
where $\alpha_1,\ldots,\alpha_{n-1}$ denote the eigenvalues of $\rho(\tau)$.
The identity holds for $\Re(s)>1$, since $|\alpha_i|=1$ for all $i$ and the
eigenvalues of $\rho(\tau^k)$ are $\alpha_i^k$.
In particular $\lambda_K(p^k)=s_r(\alpha_1,\ldots,\alpha_{n-1}) =\mathrm{tr}(\mathrm{sym}^k \rho(\tau))$, where $s_r$ is a Schur
polynomial.

Computing the logarithmic derivative, we obtain
\begin{equation*}
  \frac{L_p'}{L_p}(s,\rho_K)=\Bigl(\sum_{i=1}^{n-1}\sum_{k\geq 1}\frac{1}{k}\alpha_i^kp^{-ks}\Bigr)'
  =-\sum_{k\geq 1}\Bigl(\sum_{i=1}^{n-1}\alpha_i^k\Bigr)(\log p) p^{-ks}=
  -\log p\sum_{k\geq 1}\frac{\chi(\tau^k)}{p^{ks}}.
\end{equation*}
  Comparing this with
\eqref{logder} yields $\theta_K(p^k)= \sum\limits^{n-1}_{i=1} \alpha_i^k = \chi(\tau^k)$ which is the first assertion of the lemma.

%This shows that $\theta_K(p^k)=\chi(\tau^k)$. Moreover, we see that
%the logarithmic derivative equals $$-\log p \sum_{k=1}^{\infty}
%\mathrm{tr}(\rho(\tau)^k) p^{-ks}, $$ which is consistent with the
%formula \eqref{logder} above.

For any prime $p$, the Artin formula states that
\begin{equation*}
L_p(s,\rho_K) = \det\bigl(1 - p^{-s}\rho_K(\Frob_p) | V^{\rho_K(I_p)}\bigr)^{-1}
\end{equation*}
is the Euler factor of $L(s,\rho_K)$ at $p$, where $V=\C^{n-1}$ is the
underlying space of $\rho_K$ and $V^{\rho_K(I_p)}$ is the subspace of
$V$ where the inertia group $I_p$ acts trivially. The second assertion
now follows similarly as in the unramified case since the eigenvalues of $\rho_K(\Frob_p)$ have absolute
value $1$.
\end{proof}

\begin{remark}
The relation between the two arithmetic functions $\lambda_K$ and $\theta_K$ follows either by computing the logarithmic derivative in~\eqref{logder}, or by expressing
elementary symmetric polynomials in terms of Schur polynomials. For example $\theta_K(p)=\lambda_K(p)$ for all primes $p$.
Whereas, $\theta_K(p^2)=2\lambda_K(p^2)-\lambda_K(p)^2$, and
$\theta_K(p^3)=3 \lambda_K(p^3)-3\lambda_K(p)\lambda_K(p^2)+\lambda_K(p)^3$, and similar formulas for higher powers.
\end{remark}

If $K$ is an $S_n$-number field we can illustrate the above
construction further. The subfield $K$ of $M$ corresponds to a
subgroup $S_{n-1}$ of $S_n$, and the different embeddings
$K\hookrightarrow M$ correspond to cosets $S_{n-1}\backslash S_n$.
The group $G_{\mathfrak{p}}$ acts on the coset space
$S_{n-1}\backslash S_n$. Let $O_1,\ldots,O_k$ be the corresponding set
of orbits, ordered by size.  Then the splitting type of $p$ in $K$ is
$\tau=(\#O_1\ldots\#O_k)$. We identify $\Frob_p$ with a conjugacy
class in $S_n$ and can simply write $\Frob_p=\tau \in \T_n$ instead of
$\rho_K(\Frob_p)=\rho(\tau)$.

\subsection{Finite \'etale coverings}\label{sub:etale}
For each of the families $\FF$ considered in this paper we have a branched covering  $X\to V$ of degree $n$. The ramified locus on $V$ is given by the equation $\Delta=0$. The restriction of the covering to $V^{\Delta\neq 0}$ is finite \'etale and the normal closure has Galois group $H\hookrightarrow S_n$. 
This is a special case of~\cite{Katz} and~\cite[\S2.11]{SST} which treat the monodromy of general geometric families.

For each $f\in V(\Z)$ such that $\Delta(f) \neq 0$, the fiber $X_f$
consists of $n$ points defined over $\Z$. The individual points
themselves are elements of $\P^k(\overline{\Q})$, for some $k$
depending on $V$, but $X_f$ considered as a scheme is defined over
$\Z$. For example, when $V$ is the space of binary cubic forms, the
space $X$ is the subset of $V\times\P^1$ consisting of elements
$(f,\theta)$ such that $f(\theta)=0$, and the fiber $X_f$ can be
identified with the three roots of $f$ in $\P^1(\overline{\Q})$. When
$V$ parametrizes $S_n$-fields for $n=3$, $4$, and $5$, we have $k=1$,
$2$, and $3$, respectively. The action of $\Gal(\Q)$ on $H^0_{\et}(X_f
\times_\Q \overline{\Q},\Q_\ell)$ factors through the standard
representation of $\Gal(M_f / \Q)$.  The degree-$n$ field $K_f$ is cut
out by the stabilizer of this action.  Moreover, the cohomology of
fibers induces a lisse sheaf on $V^{\Delta\neq 0}$ of dimension $n-1$
whose stalk over each $f\in V(\Q)^{\Delta\neq 0}$ is isomorphic to
$K_f/ \Q$.  There is a monodromy action by $\pi_1(V^{\Delta\neq 0})$
and the image is $H \hookrightarrow S_n\subset \GL_{n-1}(\Q_\ell)$.

The Sato-Tate measure $\mu_{\mathrm{ST}}(\FF)$ attached to the family
is the pushforward of the Haar measure of $H$ to the space $\TT$ of
conjugacy classes of semisimple unitary matrices in $\GL_{n-1}(\C)$. By
construction it is supported on $\T_n$, for the natural inclusion
$\T_n\subset \TT$.  The Frobenius-Schur indicator is defined as
\[
i_3(\FF):=
\int_{\TT} \mathrm{tr} (t^2)\mu_{ST}(\FF)(dt).
\]
\begin{example}
If the normal closure of the covering $X\to V$ has Galois group
$H\simeq S_n$, then we say that $\FF$ is an $S_n$-family. In this case,
\begin{equation}\label{SnSatoTate}
\mu_{\mathrm{ST}}(\FF)(\{\tau\}) = 
\frac{|\tau|}{|S_n|}
\end{equation}
for every $\tau \in \T_n$, and $i_3(\FF)=1$. Indeed, the first claim
follows from the definition of $\mu_{\mathrm{ST}}(\FF)$, while the
second claim follows from a result of Frobenius--Schur \cite{FSO} (see
also \cite[Theorem 3.1 (page 151)]{FST} for the statement in modern notation). In
the terminology of~\cite{SST}, an $S_n$-family $\FF$ is homogeneous
orthogonal. This expresses that the Sato-Tate group $S_n$ acts
irreducibly on $\C^n$ and preserves a symmetric bilinear form, i.e.
the representation is real.  (In fact all irreducible representations
of $S_n$ are defined over $\Q$.)
\end{example}

For each prime $p$ we can base change to the finite field $\F_p$. If
$\Delta(f) \not \equiv 0 \pmod{p}$ then $X_f\otimes_\Z \F_p$ is
reduced and the same construction yields an action of
$\Gal(\F_p)$. Since $R_f$ is maximal at $p$, the action of Frobenius
determines the splitting type of $p$ in $K_f$. In particular the local
$L$-factor $L_p(s,\rho_{K_f})$ is uniquely determined by the base
change data of $R_f\otimes_\Z \F_p$. This is a fact that we shall use
repeatedly and which is a special case of a theorem of
Grothendieck~\cite{SGA4.5}. If moreover the covering is
$G$-equivariant for some algebraic group $G$, then the action carries
over to the reductions mod $p$ and we obtain a $G(\F_p)$-action on
$V(\F_p)^{\Delta\neq 0}$.

It is possible to prove~\cite{Katz,SST} in this generality that
\begin{equation}\label{local-frob} 
\left\{
\rho_{K_f}(\mathrm{Frob}_p):\ f\in V(\F_p)^{\Delta\neq 0}
\right\}
\end{equation}
is equidistributed as $p\to \infty$ with respect to the Sato-Tate measure $\mu_{\mathrm{ST}}(\FF)$.

\subsection{Families of degree-$n$ number fields} \label{s:families-degree-n}
In this subsection, we assume that we start with a parametric family
$\FF$ of degree-$n$ number fields as above. Recall that we abuse
notation and refer to both the family of number fields and the family
of associated $L$-functions by $\FF$.

We order the elements of $\FF$ by a height function
$h:\FF\to\R_{>0}$. When possible we choose $h(K)$ to be $|\Delta(K)|$
which is equal to the conductor of the corresponding $L$-function
$L(s,\rho_K)$. However, in some cases where it is difficult to count
elements in $\FF$ having bounded discriminant, we choose $h$ to be an
approximation of $|\Delta|$.  For $x\in \R_{\ge 1}$ we define
\[
\displaystyle \FF(x)=\{K\in\FF:h(K)<x\}.
\]
Moreover for a prime $p$, and $\tau \in \mathcal{T}_n$, define
\begin{equation*}
\begin{array}{rcl}
\displaystyle \FF^{p\nmid\Delta}(x)&=&\{K\in\FF(x):p\nmid\Delta(K)\},\\[.1in]
\displaystyle \FF^{p,\tau}(x)&=&\{K\in\FF^{p\nmid \Delta}(x):\rho_K(\Frob_p)\sim\rho(\tau)\},\\[.1in]
\displaystyle \FF^{p\mid\Delta}(x)&=&\{K\in\FF(x):p\mid\Delta(K)\}.
\end{array}
\end{equation*}
 Note that we have disjoint decompositions
 \[
  \FF(x) = \FF^{p\nmid\Delta}(x) \sqcup \FF^{p\mid\Delta}(x)
 \quad \text{ and } \quad
 \FF^{p\nmid\Delta}(x) = \bigsqcup_{\tau\in \mathcal{T}_n}\FF^{p,\tau}(x).
 \]

The main input into proving Theorem \ref{th:nic} and Theorem \ref{thmainsn} will be a counting result that estimates the number of elements in $\FF^{p,\tau}(x)$ with a power saving error term that satisfies some uniformity over $p$. More precisely, we say in the context of this paper that the Sato-Tate equidistribution holds for $\FF$ if 
there exist constants $\delta_1<\delta_0<1$ and $A,B<\infty$, and for each prime $p$ and
$\tau\in \mathcal{T}_n$ constants $0<c_{p,\tau},c_{p\mid\Delta}<1$ such that for all $x\ge 1$:
\begin{equation}\label{e:thmainsncount}
\begin{array}{rcl}
	\displaystyle |\FF^{p,\tau}(x)|&=&c_{p,\tau}|\FF(x)| +O(|\FF(x)|^{\delta_0}) + O(|\FF(x)|^{\delta_1} p^{A});\\[.1in]
\displaystyle |\FF^{p\mid\Delta}(x)|&=&c_{p\mid \Delta}|\FF(x)|+O(|\FF(x)|^{\delta_0})+O(|\FF(x)|^{\delta_1} p^{B}).
\end{array}
\end{equation}
\begin{remark}

The remainder terms in \eqref{e:thmainsncount} are all dominated by $O(|\FF(x)|^{\delta_0}p^{\max(A,B)})$ which would be sufficient for our purpose to establish the statistics of low-lying zeros for some positive support. However, we write the formulas~\eqref{e:thmainsncount} in this more precise form because this is what the proof naturally produces for geometric families and this yields an improved support. %In general one could try to refine the estimates even more by having a finite sum of error terms of the form $O(|\FF(x)|^{\delta_i}p^{A_i})$ with $1>\delta_0>\delta_1>\delta_2\ldots$, ordered by descending powers of $|\FF(x)|$.
\end{remark}

The constants $c_{p,\tau}$ in fact determine the unramified part of the
probability measure $\mu_p(\FF)$ defined in~\cite[Conj.1]{SST}.
The ramified part of the measure $\mu_p(\FF)$ is more complicated and will be discussed in Section~\ref{s:mass}.
 It is clear that for every prime $p$,
\[
c_{p \mid \Delta} + \sum_{\tau \in \T_n} c_{p,\tau} = 1.
\]
Let $\TT:=(S^1)^{n-1}/S_{n-1}$, which can be identified with the set of conjugacy
classes of semisimple matrices in the compact unitary group $U_{n-1}$. 
The standard representation $S_n \to U_{n-1}$ induces a natural inclusion $\T_n \subset \TT$. Concretely, say that $\tau\in \T_n$ corresponds to the
partition $(f_1f_2\ldots f_k)$. Then we form the $n$-tuple of $f_i$-th roots of unity, for $1\le i \le k$, which is an element of $(S^1)^n$, and we remove
the trivial root $1$ once, to obtain an element of $\TT$.

Up to a scalar, $\mu_p(\FF)$ is the counting measure on the set~\eqref{local-frob} of splitting types modulo $p$. Precisely, the unramified part $\mu_p(\FF)_{|\TT}$ is supported on $\T_n\subset \TT$, and for every $\tau\in \T_n$,
\[
 \mu_p(\FF)(\{\tau\}) = c_{p,\tau} .
\]
Thus the unramified part $\mu_p(\FF)_{|\TT}$ is a measure of total mass
$\mu_p(\FF)(\TT) = 1-c_{p\mid\Delta}$.

For the
families $\FF$ obtained by application of a square-free sieve to $V(\Z)$, we have that $c_{p,\tau}$ is given by a $p$-adic density.
In all such cases we have the identity
\[
\frac{c_{p,\tau}}{1-c_{p|\Delta}}
=
\frac{|V(\F_p)^\tau|}{|V(\F_p)^{\Delta\neq 0}|},
\] 
where $V(\F_p)^\tau$ is the set of all elements in $V(\F_p)$ having splitting type $\tau$. Thus in view of \S\ref{sub:etale}, and the fact that $c_{p|\Delta}\to 0$, we have
\[
\mu_p(\FF)_{|\TT} \rightharpoonup \mu_{\mathrm{ST}}(\FF).
\]
Equivalently for each $\tau\in \T_n$, we have that $c_{p,\tau}$ converges to $\mu_{\mathrm{ST}}(\FF)(\{\tau\})$ as $p\to \infty$.

One quantity that is especially important in the study of $\FF$ is the average trace of unramified Frobenius. With the above notation it can be expressed as
\[
t_{\FF}(p) := \sum_{\tau\in \T_n} c_{p,\tau}\chi(\tau) %\frac{|\tau|}{|S_n|}
=\int_{\TT} \mathrm{tr}(t) \mu_p(\FF)(dt).
\]
We have that
$\dfrac{t_{\FF}(p)}{1-c_{p | \Delta}}$ is a sum of $p$-Weil numbers with integer weights. By orthogonality of characters, it follows
that if $H$ acts without non-zero fixed vector in $\C^{n-1}$, then 
$t_\FF(p)=O(\frac{1}{p})$.

\subsection{The $1$-level density of low-lying zeros of $S_n$-families}

In this subsection we compute the $1$-level density of the low-lying zeros of the Artin $L$-functions of the families considered in
Sections~\ref{s:families}, \ref{s:other}, and \ref{s:n-ary}. We do this calculation in the ``traditional'' way, and we explain at the same
time how the main term can be found conceptually from the Sato-Tate measure as in~\cite{SST,ShinTemplier}.

The above families are $S_n$-families, and thus the Sato-Tate measure $\mu_{\mathrm{ST}}(\FF)$ is given by~\eqref{SnSatoTate}. In fact, in Sections~\ref{s:families},~\ref{s:other} and~\ref{s:n-ary} we will establish that each of these $S_n$-families satisfy the Sato-Tate equidistribution~\eqref{e:thmainsncount} (with constants  $\delta_0,\delta_1,A,B$) and that $|\FF(x)|\asymp x^\theta$ for some $\theta>0$. Furthermore, in these cases we will also prove the following regarding the constants $c_{p,\tau}$ and $c_{p\mid\Delta}$: For any prime $p$ and $\tau\in \T_n$, we have
\begin{equation}\label{cptau}
\displaystyle c_{p,\tau}=\frac{|\tau|}{|S_n|}+O\Bigl(\frac1{p}\Bigr),
\end{equation}
where $|\tau|$ denotes the size of the conjugacy class $\tau$ in $S_n$. In particular $c_{p \mid \Delta} =
O\bigl(\frac1{p}\bigr)$ since $\sum\limits_{\tau\in \T_n} \frac{|\tau|}{|S_n|}=1$ and also we recover that $t_\FF(p)=O(\frac1p)$ since
$\sum\limits_{\tau\in \T_n} \chi(\tau)\frac{|\tau|}{|S_n|}=0$.

One reason to refer to these families as $S_n$-families is that a consequence of~\eqref{cptau} is that most $K\in \FF(x)$ are $S_n$-fields in the sense that $S_n$ is the Galois group of their normal closure. Indeed this follows in the same way as Hilbert's irreducibility theorem by applying a sieve to construct Frobenius elements which are $n$-cycles and transpositions.

Let $f$ be a fixed Paley-Wiener function as in the beginning of \S\ref{s:setup}. We are interested in evaluating
\begin{equation*}
	\lim_{x\to\infty}\frac{1}{|\FF(x)|}\displaystyle\sum_{K\in\FF(x)}\displaystyle\sum_j f\bigg(\frac{\gamma^{(j)}_K\cL}{2\pi}\bigg),
\end{equation*}
where $\cL$ will be picked so that we capture the statistics of the low-lying zeros. The natural choice for $\cL$ is $\log C_K$, where $C_K$ is the conductor of $L(s,\rho_K)$, because we expect the lowest zeros of $L(s,\rho_K)$ to be at height around $\frac{2\pi}{\log C_K}$. However, we pick $\cL=\cL(x)$ to be
\begin{equation}\label{eqL}
\cL:=\frac{1}{|\FF(x)|}\sum_{K\in\FF(x)}\log C_K,
\end{equation} 
the average of these natural choices. In view of the counting asymptotic in~\eqref{e:thmainsncount}, we have 
\begin{equation}\label{sizeofl}
\cL=(1+o(1)) \log x  \qquad \text{ as } x\to \infty
\end{equation}
if $h(K)$ equals or closely approximates $C_K$, that is, if the family is ordered by a quantity that closely approximates the absolute discriminant. In all our examples, this will be true.

\begin{theorem}\label{th:low-lying}
Let $\FF$ be one of the $S_n$-families of Sections~\ref{s:families},~\ref{s:other} and~\ref{s:n-ary}.
  If $f$ is a function whose Fourier transform is smooth and has support in $[-\alpha,\alpha]$ for 
\begin{equation*}
\alpha< \min\Bigl(2\theta(1-\delta_0),\frac{2\theta(1-\delta_1)}{2C+1}\Bigr),
\end{equation*}
where $C:=\max(A,B)$, then
\begin{equation}\label{low-lying-limit}
\lim_{x\to\infty}\frac{1}{|\FF(x)|}\displaystyle\sum_{K\in\FF(x)}\displaystyle\sum_j f\bigg(\frac{\gamma^{(j)}_K\cL}{2\pi}\bigg)=\widehat{f}(0)-\displaystyle\frac{f(0)}{2}.
\end{equation}
\end{theorem}

\begin{remark}
 The Sato-Tate measure $\mu_{\mathrm{ST}}(\FF)$ will be different from~\eqref{SnSatoTate} if $\FF$ is not an $S_n$-family. For example, in Section~\ref{s:quaternion} we investigate a family of quaternionic extensions where the Sato-Tate group is $Q_8 \subset \GL_2(\C)$. Also, in Section~\ref{s:mixed} we investigate families where the Sato-Tate group is $D_4 \subset \GL_2(\C)$ as well as the reducible examples $C_3 \subset \GL_2(\C)$ and $D_4 \subset \GL_3(\C)$. In those cases the right-hand side of~\eqref{low-lying-limit} should be replaced by 
$
\widehat f(0) - i_{3}(\FF) \frac{f(0)}{2},
$
where $i_3(\FF)$ is the Frobenius-Schur indicator of $\mu_{\mathrm{ST}}(\FF)$.
Note that $\alpha<1$ in all these examples.
\end{remark}
\begin{proof}
	Without loss of generality we may assume that $f$ is even because $L(s,\rho_K)$ is self-dual and thus $\gamma$ is a zero if and only if $-\gamma$ is a zero. We use
\eqref{eqexplicit} to write the above as the limit as $x\to\infty$ of
\begin{equation}\label{eqexpused}
\frac1{|\FF(x)|}\sum_{K\in\FF(x)}\biggl( \frac1{2\pi}\int_{-\infty}^\infty f\Big(\frac{t\cL}{2\pi}\Big)(\log
C_K+O(1))dt- 
\frac{2}{\cL}\sum_{m=1}^\infty\frac{\theta_K(m)\Lambda(m)}{\sqrt{m}}\widehat{f}\Bigl(\frac{\log
  m}{\cL}\Bigr)
  +\displaystyle\sum_j f\bigg(\frac{r^{(j)}_K\cL}{2\pi}\bigg) \biggr).
\end{equation}
Here the contribution from the sum over poles of $L(s,\rho_K)$ is negligible (i.e. $o(1)$) because the test function $f$ is assumed to be of rapid decay and the only possible locations for poles of $L(s,\rho_K)$ are at the zeros of $\zeta(s)$ (cf.\ \eqref{Dedekind-relation}). Furthermore, since $C_K\to\infty$ as $h(K)\to\infty$, we can evaluate the limit of
the first part of \eqref{eqexpused} to be
\begin{equation}\label{limit1}
\begin{array}{rcl}
\displaystyle\frac1{|\FF(x)|}
\displaystyle\sum_{K\in\FF(x)}
\displaystyle\frac1{2\pi}\int_{-\infty}^\infty f\Big(\frac{t\cL}{2\pi}\Big)(\log C_K+O(1))dt &=&
\displaystyle\frac1{\cL|\FF(x)|}
\displaystyle\sum_{K\in\FF(x)} \log C_K
\displaystyle\int_{-\infty}^\infty f(t)(1+o(1))dt\\[.2in]
&\longrightarrow&\displaystyle\int_{-\infty}^\infty f(t)dt = \widehat{f}(0).\\[.2in]
\end{array}
\end{equation}
To evaluate the limit of the second part of 
\eqref{eqexpused}, we note that
\begin{equation}\label{finalcont}
\frac{2}{\cL|\FF(x)|}\sum_{K\in\FF(x)}\sum_{m=1}^\infty\frac{\theta_K(m)\Lambda(m)}{\sqrt{m}}\widehat{f}\Bigl(\frac{\log
  m}{\cL}\Bigr)=\frac{2}{\cL|\FF(x)|}\sum_{p,k\geq 1}\frac{\log p}{p^{k/2}}
\widehat{f}\Bigl(\frac{k\log p}{\cL}\Bigr)\sum_{K\in\FF(x)}\theta_K(p^k),
\end{equation}
where the change in the order of summation is justified because
$\widehat{f}$ has compact support, and hence the sums over $m$, $p$ and $k$ are
finitely supported. We write the right-hand side of the above equation as the limit as $x\to\infty$ of
$\CS_1+\CS_2+\CS_3+\CS_\ram$, where
\begin{equation}\label{eqs}
\begin{array}{rcl}
\CS_1&:=&
\displaystyle\frac{2}{\cL|\FF(x)|}\displaystyle\sum_{p}
\displaystyle\frac{\log p}{\sqrt{p}}\widehat{f}\Bigl(
\displaystyle\frac{\log p}{\cL}\Bigr)
\displaystyle\sum_{K\in\FF^{p\nmid\Delta}(x)}\theta_K(p);
\\[.25in]\CS_2&:=&
\displaystyle\frac{2}{\cL|\FF(x)|}\displaystyle\sum_{p}
\displaystyle\frac{\log p}{p}\widehat{f}\Bigl(
\displaystyle\frac{2\log p}{\cL}\Bigr)
\displaystyle\sum_{K\in\FF^{p\nmid\Delta}(x)}\theta_K(p^2);
\\[.25in]\CS_3&:=&
\displaystyle\frac{2}{\cL|\FF(x)|}\displaystyle\sum_{p,k\geq 3}
\displaystyle\frac{\log p}{p^{k/2}}\widehat{f}\Bigl(
\displaystyle\frac{k\log p}{\cL}\Bigr)
\displaystyle\sum_{K\in\FF^{p\nmid\Delta}(x)}\theta_K(p^k);
\\[.25in]\CS_\ram&:=&
\displaystyle\frac{2}{\cL|\FF(x)|}\displaystyle\sum_{p,k\geq 1}
\displaystyle\frac{\log p}{p^{k/2}}\widehat{f}\Bigl(
\displaystyle\frac{k\log p}{\cL}\Bigr)
\displaystyle\sum_{K\in\FF^{p\mid\Delta}(x)}\theta_K(p^k).
\end{array}
\end{equation}
To evaluate the sums \eqref{eqs}, we begin by writing
\begin{equation}\label{eqs1}
\begin{array}{rcl}
\CS_1&=&\displaystyle\frac{2}{\cL|\FF(x)|}\displaystyle\sum_{p}
\displaystyle\frac{\log p}{\sqrt{p}}\widehat{f}\Bigl(
\displaystyle\frac{\log p}{\cL}\Bigr)
\displaystyle\sum_{K\in\FF^{p\nmid\Delta}(x)}\theta_K(p)
\\[.25in]&=&
\displaystyle\frac{2}{\cL|\FF(x)|}\displaystyle\sum_{p}
\displaystyle\frac{\log p}{\sqrt{p}}\widehat{f}\Bigl(
\displaystyle\frac{\log p}{\cL}\Bigr)\sum_{\tau\in \mathcal{T}_n}|\FF^{p,\tau}(x)|\chi(\tau)
\\[.25in]&=&
\displaystyle\frac{2}{\cL|\FF(x)|}\displaystyle\sum_{p}
\displaystyle\frac{\log p}{\sqrt{p}}\widehat{f}\Bigl(
\displaystyle\frac{\log p}{\cL}\Bigr)\Bigl(
t_\FF(p) |\FF(x)| +
O(|\FF(x)|^{\delta_0})+O(|\FF(x)|^{\delta_1} p^A)\Bigr)
\\[.25in]&=&
O\Bigl(\displaystyle\frac{1}{\cL}\Bigr)+O\biggl(\displaystyle\frac{e^{\frac{\cL\alpha}{2}}}{|\FF(x)|^{1-\delta_0}\cL}\biggr)+O\biggl(\displaystyle\frac{e^{\cL\alpha(A+\frac12)}}{|\FF(x)|^{1-\delta_1}\cL}\biggr),
\end{array}
\end{equation}
where the final equality follows by computing the third line of
\eqref{eqs1} using the fact that since $\widehat{f}$ is supported on
$[-\alpha,\alpha]$, the sum over $p$ can be restricted to the range
$p\leq e^{\cL\alpha}$; the bounds follow from Lemma \ref{lemap},
\eqref{e:thmainsncount}, and the fact that
$t_\FF(p)=O(\frac{1}{p})$. Similarly, we have
\begin{equation}\label{eqs23ram}
\begin{array}{rcl}
\CS_2&=&
\displaystyle\frac{2}{\cL}\displaystyle\sum_{p}
\displaystyle\frac{\log p}{p}\widehat{f}\Bigl(
\displaystyle\frac{2\log p}{\cL}\Bigr)
\sum_{\tau \in \mathcal{T}_n} \chi(\tau^2) \frac{|\tau|}{|S_n|}
+O\Bigl(\frac{1}{|\FF(x)|^{1-\delta_0}\cL}\Bigr)+O\biggr(\frac{e^{\frac{\cL\alpha A}{2}}}{|\FF(x)|^{1-\delta_1}\cL}\biggl)+o(1),\\[.2in]
\CS_3&=&
O\biggl(\displaystyle\frac{e^{\frac{\cL\alpha}{3}(A-\frac12)}}{|\FF(x)|^{1-\delta_1}\cL}\biggr)+o(1),\\[.2in]
\CS_\ram&=&
O\biggl(\displaystyle\frac{1}{\cL|\FF(x)|}\displaystyle\sum_{p,k\geq 1}
\frac{\log p}{p^{k/2}}\widehat{f}\Bigl(
\displaystyle\frac{k\log p}{\cL}\Bigr)\Bigl(\frac{|\FF(x)|}{p}+|\FF(x)|^{\delta_0}+|\FF(x)|^{\delta_1} p^B\Bigr)\biggr)\\[.25in]
&=&O\biggl(\displaystyle\frac{e^{\frac{\cL\alpha}{2}}}{|\FF(x)|^{1-\delta_0}\cL}\biggr)+O\biggl(\displaystyle\frac{e^{\cL\alpha(B+\frac12)}}{|\FF(x)|^{1-\delta_1}\cL}\biggr)+o(1).
\end{array}
\end{equation}
Therefore, in the limit $x\to\infty$, the only possible main term contribution to the right-hand side of \eqref{finalcont} is from the sum $\CS_2$.

% It is easy to see that the main term contribution from the sum $\CS_1$ is $0$. Indeed, $\chi$ is the character of the nontrivial irreducible representation $\rho$ and the quantity
% \[
% \sum_{\tau \in \mathcal{T}_n} \chi(\tau) \frac{|\tau|}{|S_n|}
% \]
% is the inner product of $\chi$ with the trivial representation. It is $0$ by the orthogonality of characters.
The main term of $\CS_2$ includes the  sum
\begin{equation}\label{sni3}
	i_3(\FF)=\sum_{\tau \in \mathcal{T}_n} \chi(\tau^2) \frac{|\tau|}{|S_n|}=1
\end{equation}
which is the Frobenius-Schur indicator of the representation $\rho:S_n\to \GL_{n-1}(\C)$.
 Therefore, the main term contribution from $\CS_2$ is
\begin{equation}\label{eqtempeval1}
\lim_{x\to\infty}\displaystyle\frac{2}{\cL}\displaystyle\sum_{p}
\displaystyle\frac{\log p}{p}\widehat{f}\Bigl(
\displaystyle\frac{2\log p}{\cL}\Bigr)=\int_{0}^\infty\widehat{f}(t)dt,
\end{equation}
where the equality follows from the prime number theorem and integration by parts.
Since the right-hand side of \eqref{eqtempeval1} is $f(0)/2$ by Fourier inversion, we have
\begin{equation}\label{eqfinalsncomp}
	\frac{1}{|\FF(x)|}\displaystyle\sum_{K\in\FF(x)}\displaystyle\sum_j f\bigg(\frac{\gamma^{(j)}_K\cL}{2\pi}\bigg)=
\widehat{f}(0)-\displaystyle\frac{f(0)}{2}+O\biggl(\displaystyle\frac{e^{\frac{\cL\alpha}{2}}}{|\FF(x)|^{1-\delta_0}\cL}\biggr)+
O\biggl(\displaystyle\frac{e^{\cL\alpha(C+\frac12)}}{|\FF(x)|^{1-\delta_1}\cL}\biggr)+o(1),
\end{equation}
where $C=\max(A,B)$. This indicates the symplectic symmetry type for the low-lying zeros of $L$-functions in these families.

Finally, we assume there exists $\theta>0$ such that $|\FF(x)|\asymp x^\theta$. This will be true in all the $S_n$-families that we consider. Therefore, by \eqref{sizeofl}, we have
$$
\displaystyle\frac{e^{\frac{\cL\alpha}{2}}}{|\FF(x)|^{1-\delta_0}\cL}+\displaystyle\frac{e^{\cL\alpha(C+\frac12)}}{|\FF(x)|^{1-\delta_1}\cL}=
x^{\frac{\alpha}{2}-\theta(1-\delta_0)+o(1)} + x^{\alpha(C+\frac12)-\theta(1-\delta_1)+o(1)}.
$$
We conclude that the error terms in \eqref{eqfinalsncomp} will be bounded by $o(1)$ whenever 
$$
\alpha< \min\Bigl(2\theta(1-\delta_0),\frac{2\theta(1-\delta_1)}{2C+1}\Bigr).
$$
This concludes the proof.
\end{proof}

\subsection{Rank of families}\label{sub:averageFrob}
Recall that we used the estimate $t_{\FF}(p) = O(\frac{1}{p})$ in the proof of Theorem~\ref{th:low-lying}, specifically in the estimation of $\CS_1$.
The interpretation is that the rank of these number field families is zero, namely
\[
\lim_{y\to \infty} \frac{1}{y} \sum_{p<y} -t_{\FF}(p) p^{\frac12} \log p =0.
\]

Examples of families where the rank is non-zero are families of elliptic curves in which case $t_{\FF}(p)$ is a sum of Weil numbers
of half-integer weights. For number field families the weights are always integer and thus the rank is always zero. This is consistent with the belief that each irreducible Artin $L$-function is non-vanishing at the central point unless the epsilon factor is $-1$ in which case it is believed to vanish with order one.

\subsection{Other indicators of $S_n$-families} 
The other indicators defined in~\cite{SST}, that is
\[
i_1(\FF) := 
\int_{\TT} |\mathrm{tr} (t)|^2\mu_{ST}(\FF)(dt),\quad 
i_2(\FF) := 
\int_{\TT} \mathrm{tr} (t)^2\mu_{ST}(\FF)(dt),\\
\]
are not used in the proof of Theorem~\ref{th:low-lying}. They satisfy  $i_1(\FF)=1$ and $i_2(\FF)=1$ for $S_n$-families, expressing the fact that $S_n \subset \GL_{n-1}(\C)$ acts irreducibly and is self-dual. 

For $S_n$-families of Artin representations parametrized
geometrically, one can establish by a sieve that most $K \in \FF(x)$
are $S_n$-fields.  In particular most Artin $L$-functions
$L(s,\rho_K)$ in an $S_n$-family are irreducible and self-dual
orthogonal. The argument is unconditional taking advantage of the
underlying algebraic structure and the finiteness of the Galois
group. This is to be compared with~\cite{SST} for general homogeneous
families with $i_1(\FF)=i_2(\FF)=i_3(\FF)=1$, where it is explained
that this would also follow from the GRH by detecting\footnote{The GRH
  is needed in~\cite{SST} to truncate the Euler product to a small
  number of primes so that one can apply the quantitative Sato-Tate
  equidistribution for the family.} the simple pole at $s=1$ of the
Rankin-Selberg product $L(s,\rho_K \times \tilde \rho_K)$ which
implies irreducibility and similarly for $L(s,\mathrm{sym}^2 \rho_K)$
which implies orthogonality.

\section{Parametrized families of cubic, quartic, and quintic fields}\label{s:families}
In this section, we consider parametrized families of cubic, quartic,
and quintic fields. These families are constructed from certain
prehomogeneous representations. A representation $V$ of $G$ is said to
be {\it prehomogeneous} if $V$ has a Zariski-dense
$G$-orbit. Irreducible prehomogeneous representations of reductive
groups over $\C$ were classified by Sato--Kimura
\cite{SatoKimura}. The rational orbits of these representations were
studied in the work of Wright and Yukie \cite{Wright-Yukie}, who also
explained their connection to field extensions. For our applications
we need an interpretation of the $\Z$-orbits of these
representations. For the representation $\Sym^3(\Z^2)$ of $\GL_2(\Z)$,
such an interpretation is due to Levi \cite{Levi} and Delone--Faddeev
\cite{DF}, who show that the orbits having nonzero discriminant
correspond bijectively to reduced cubic rings over $\Z$. This
correspondence was refined by Gan--Gross--Savin \cite{GGS}, and shown
also to hold for orbits having discriminant $0$.  Analogous
parametrizations of quartic and quintic rings over $\Z$ are developed
by Bhargava in his landmark works \cite{Bquartic} and \cite{Bquintic},
respectively. His work also naturally recovers the cubic case, and
provides a geometric view of it. We now briefly describe the parts of
this theory necessary for us.

For $n=3$, $4$, and $5$, consider the space of degree-$n$ rings over
$\Z$ along with the additional data of a resolvent ring. That is,
consider the space of pairs $(R_1,R_2)$, where $R_1$ is a degree-$n$
ring, and $R_2$ is a resolvent ring of $R_1$. The resolvent ring of a
cubic ring over $\Z$ is simply the unique quadratic ring having the
same discriminant. For the definitions of resolvent rings of quartic
and quintic rings over $\Z$, see \cite[\S2.3]{Bquartic} and
\cite[\S5]{Bquintic}, respectively. Bhargava proves that this space is
parametrized by $G_n(\Z)$-orbits on $V_n(\Z)$, for $n=3$, $4$, and
$5$, where $G_n$ is a reductive group and $V_n$ is a prehomogeneous
representation of $G_n$. The condition that a $G_n(\Z)$-orbit of $v\in
V_n(\Z)$ corresponds to a maximal ring is given by congruence
conditions on $V_n(\Z)$. Bhargava also shows that a maximal ring has a
{\it unique} resolvent ring!

An element in $V(\Z)$ is said to be {\it
  $S_n$-irreducible} if it corresponds to an order in an
$S_n$-field. Let $V_n(\Z)^\smax$ denote the set of
$S_n$-irreducible elements of $V_n(\Z)$ that correspond to maximal rings.
Therefore, the set of $G_n(\Z)$-orbits on
$V_n(\Z)^\smax$ can be considered to be a parametrized family of
degree-$n$ fields. The ring of relative invariants for the action of
$G_n$ on $V_n$ is freely generated by one invariant, which we call the
{\it discriminant}. The discriminant of $v\in V_n(\Z)$ is equal to the
discriminant of the ring corresponding to $v$. Thus, to count the
number of degree-$n$ fields having discriminant bounded by $x$, it
suffices to count the number of $G_n(\Z)$-orbits on $V_n(\Z)^\smax$ with
discriminant bounded by $x$. This is carried out by Davenport and
Heilbronn \cite{DH} in the case $n=3$, and by Bhargava
\cite{B1},\cite{B2} in the cases $n=4,5$, respectively.

The condition that $v\in V_n(\Z)$ is $S_n$-irreducible is not a local
condition and is imposed in two steps. First, the cuspidal regions of
the fundamental domain $G_n(\Z)\backslash V_n(\R)$ containing integral
points corresponding to non $S_n$-irreducible rings are cut off. Next,
the main ball is shown to contain predominantly $S_n$-irreducible
points. The latter step follows from an application of Hilbert
irreducibility; a power saving may be obtained using the Selberg
sieve. The condition that $v\in V_n(\Z)$ corresponds to a maximal ring
is a local condition. A ring $R$ that is a finitely generated
$\Z$-module is maximal if and only if it is maximal at every prime
$p$, i.e., $R\otimes\Z_p$ is maximal over $\Z_p$. For $n=3$, $4$, and
$5$, degree-$n$ ring extensions of $\Z_p$ are classified by
$G_n(\Z_p)$-orbits on $V_n(\Z_p)$. We denote the set of elements in
$V_n(\Z_p)$ corresponding to maximal $\Z_p$-extensions by
$V_n(\Z_p)^\max$. For $n=3$, $4$, and $5$, it is proven in \cite{DH},
\cite{Bquartic}, and \cite{Bquintic}, respectively, that
$V_n(\Z_p)^\max$ can be described by congruence conditions modulo
$p^2$ on $V_n(\Z_p)$.

%For example, in the case $n=3$,
%the prime $p$ splits completely in the field $K$ corresponding to $v$
%if and only if the cubic extension over $\Z_p$ corresponding to $v$ is
%$\Z_p^3$; all such elements $v$ consist of a single
%$\GL_2(\Z_p)$-orbit.

For our purpose of computing the symmetry type of the low-lying zeros
of zeta functions arising from degree-$n$ fields, we need to also
count the number of degree-$n$ fields with prescribed splitting type
at a fixed prime $p$. This is done as follows: consider the injection
$V_n(\Z)\to V_n(\Z_p)$. The splitting of $p$ in the field
corresponding to $v$ is determined by the $G_n(\Z_p)$-orbit of $v$ in
$V_n(\Z_p)$.  Furthermore, the set of all $v\in V_n(\Z_p)^\max$ having
a fixed splitting type consists of finitely many
$G_n(\Z_p)$-orbits. Given a splitting type $\tau$, we denote the set
of elements in $V_n(\Z_p)$ corresponding to $\tau$ by
$V_n(\Z_p)^\tau$. For unramified splitting types $\tau$, every element
in $V_n(\Z_p)^\tau$ is maximal.
%and the set of maximal elements in $V_n(\Z_p)^\tau$ by $V_n(\Z_p)^{\tau,\max}$.
Next consider the reduction modulo $p$ map $V_n(\Z_p)\to
V_n(\F_p)$. In fact, the splitting type $\tau$ of $v\in V_n(\Z_p)$ is
determined by the image $\bar{v}$ of $v$ in $V_n(\F_p)$. Moreover, the
set of all $\bar{v}\in V_n(\F_p)$ corresponding to a fixed splitting
type consists of a {\em single} $G_n(\F_p)$-orbit. We will use the map
$V(\Z)\to V(\Z_p)$ as well as the map $V(\Z)\to V(\F_p)$; the first is
necessary to detect maximality at $p$, while the second suffices to
determine the splitting type at an unramified prime $p$.

From this, it is possible to see why we expect Equation \eqref{cptau}
to be true. Let $\tau$ denote a fixed splitting type, $O_\tau\subset
V_n(\F_p)$ denote the corresponding $G(\F_p)$-orbit, and let
$\F_p(\tau)$ denote the corresponding extension of $\F_p$. We expect
that
\begin{equation}\label{eqh1}
  \begin{array}{rcl}
  |\FF^{p,\tau}(x)|&\sim&
  \displaystyle\frac{\Vol(V_n(\Z_p)^{\tau})}{\Vol(V_n(\Z_p)^\max)}\cdot|\FF(x)|
  \\[.2in]&\sim&
  \displaystyle\frac{\Vol(V_n(\Z_p))}{\Vol(V_n(\Z_p)^\max)}\cdot
  \frac{|O_\tau|}{|V(\F_p)|}\cdot|\FF(x)|.
  \end{array}
\end{equation}
Next we expect the estimate
\begin{equation}\label{eqh2}
\displaystyle\frac{\Vol(V_n(\Z_p))}{\Vol(V_n(\Z_p)^\max)}=1+O\Bigl(\frac{1}{p^2}\Bigr)
\end{equation}
to hold since a proportion of roughly $1/p^2$ of elements in $V(\Z_p)$
are nonmaximal. Indeed, if a ring $R$ is nonmaximal at $p$ then $p^2$
divides the discriminant of $R$.
%and that a proportion of at most $~1/p$ of elements in $V(\Z_p)^\tau$
%are nonmaximal for any splitting type $\tau$.

Thus, it is only required to check that $|O_\tau|/|V(\F_p)|=
|\tau|/|S_n|+O(1/p)$, for $\tau\in \T_n$, where we are abusing
notation by considering $\tau$ both as a splitting type and as the
corresponding conjugacy class in $S_n$. Let $\bar{v}\in O_\tau$ denote
any element, and let $\sigma_\tau\in S_n$ denote any element in the
conjugacy class $\tau$. Our representations $(G,V)$ satisfy the
property
$\Stab_{G(\F_p)}(\bar{v})\cong\Aut(\F_p(\tau))\cong\Stab_{S_n}(\sigma_\tau)$
(see, for example, \cite[Theorem 6]{BSWPre}). By two applications of
the orbit-stabilizer formula, we obtain
\begin{equation}\label{eqh3}
\displaystyle\frac{|O_\tau|}{|V(\F_p)|}=
\frac{|G(\F_p)|}{|\Stab_{G(\F_p)}(\bar{v})||V(\F_p)|}=
\frac{1}{|\Stab_{S_n}(\sigma_\tau)|}+O\Bigl(\frac1{p}\Bigr)=
\frac{|\tau|}{|S_n|}+O\Bigl(\frac1{p}\Bigr),
\end{equation}
as required. To see why we expect $c_{p\mid\Delta}=O(\frac 1p)$, note
that $p$ ramifies in the field corresponding to $v\in V(\Z)$ if and
only if the discriminant of $\bar{v}\in V(\F_p)$ is zero. Furthermore,
the number of elements in $V(\F_p)$ having discriminant 0 is bounded
by $O(|V(\F_p)|/p)$.

For $n=3$, $4$, and $5$, the estimates \eqref{eqh1},
\eqref{eqh2}, \eqref{eqh3} are known to be true. Indeed, in the rest
of this section, we give detailed references and explain how to obtain
Sato-Tate equidistribution and \eqref{cptau} for the families of
cubic, quartic, and quintic fields, and describe the error terms that
we obtain.  The purpose of the above discussion is to give a heuristic
explanation for why we expect \eqref{cptau} to hold in greater
generality.

\subsection{The family of cubic fields}
Let $V$ denote the space of binary cubic forms. The group $G=\GL_2$
acts on $V$ via the twisted action
$$g\cdot f(x,y):=\frac{1}{\det g}f((x,y)\cdot g).$$ A result of
Delone--Faddeev \cite{DF}, refined by Gan--Gross--Savin \cite{GGS},
states that isomorphism classes of cubic orders is parametrized by
$G(\Z)$-orbits on $V(\Z)$. The congruence conditions defining
maximality is a result of Davenport and Heilbronn \cite{DH}.
\begin{theorem}\label{thcubicrings}
\begin{itemize}
\item[{\rm (1)}] There is a natural bijection between the set of
  $G(\Z)$-equivalence classes of integral binary cubic forms and the
  set of isomorphism classes of cubic rings. A cubic ring
  corresponding to the $G(\Z)$-orbit of $f$ is an order if and only if
  $f(x,y)$ is irreducible over $\Q$.

\item[{\rm (2)}] A cubic order corresponding to $f\in V(\Z)$ fails to
  be maximal at $p$ if either $f$ is a multiple of $p$ or if some
  $G(\Z)$-translate $ax^3+bx^2y+cxy^2+dy^3$ of $f(x,y)$ satisfies
  $p^2\mid a$ and $p\mid b$.
\end{itemize}
\end{theorem}
The discriminant $\Delta$ of a binary cubic form is
$G$-invariant. Furthermore, the discriminant of $f$ equals the
discriminant of the cubic ring corresponding to $f$.

\begin{example}
 If $f=x^3+7y^3$ then the cubic ring is $R=\Z[x]/(x^3+7)$. The form
 $f$ is irreducible over $\Q$ and $R$ is the maximal order in the
 field $\Q(\sqrt[3]{7})$. $($To check maximality, it is only necessary
 to verify the conditions in {\rm \cite[Lemma 2.10]{BBP}}.$)$ Let
 $p=3$. Then the form $\bar{f}(x,y)$ is irreducible. Therefore, the
 splitting type of $p$ in $\Q(\sqrt[3]{7})$ is $(3)$.
\end{example}

To count cubic fields, we directly use \cite[Theorem 1.3]{TaTh}.
\begin{theorem}\label{thcubiccount}
Let $\FF$ be the parametrized family of cubic $S_3$-fields ordered by
discriminant.  For any prime $p$, conjugacy class $\tau\in
\mathcal{T}_n$ and $\epsilon>0$, we have
\begin{equation}
\begin{array}{rcl}
\displaystyle
|\FF(x)|&=&\displaystyle\frac{1}{3\zeta(3)}x+O(x^{5/6}),\\[.2in]
\displaystyle |\FF^{p,\tau}(x)|&=&\displaystyle
\frac{c_{p,\tau}}{3\zeta(3)}x +O(x^{5/6})
+O_{\epsilon}(x^{7/9+\epsilon} p^{8/9}),\\[.2in] \displaystyle
|\FF^{p\mid\Delta}(x)|&=&\displaystyle
\frac{c_{p\mid\Delta}}{3\zeta(3)}x +O(x^{5/6})
+O_{\epsilon}(x^{7/9+\epsilon} p^{16/9}),
\end{array}
\end{equation}
where $c_{p,\tau}=\displaystyle\frac{|\tau|}{6}\frac{p^2}{p^2+p+1}$
and $c_{p\mid\Delta}=\displaystyle\frac{p+1}{p^2+p+1}$.
\end{theorem}
This concludes the proof of the Sato-Tate equidistribution for the
family $\FF$ of cubic fields, namely the equation
\eqref{e:thmainsncount} with $\delta_0=5/6$ and
$\delta_1=7/9+\epsilon$, together with \eqref{cptau}. Note that in the
two estimates in~\eqref{e:thmainsncount} the exponents $A=8/9$ and
$B=16/9$ differ. Also the exponent $\delta_0=5/ 6$ is sharp
by~\cite{BST} and \cite{TaTh}, which independently establish a
secondary main term for the counting function of cubic fields. We
obtain the bound on the support to be $\alpha<4/41$ in
Theorem~\ref{th:low-lying}.

\subsection{The family of quartic fields}

Let $V=2\otimes\Sym^2(3)$ denote the space of pairs of ternary
quadratic forms. We represent elements in $V$ by a pair of symmetric
$3\times 3$-matrices $A$ and $B$. The group $G=\GL_2\times\SL_3$ acts
on $V$ via the action
$$(g_2,g_3)\cdot (A,B):=(g_3^tAg_3,g_3^tBg_3)\cdot g_2^t.$$ A result
of Bhargava \cite{Bquartic} states that isomorphism classes of pairs
$(Q,C)$, where $Q$ is a quartic ring and $C$ is a cubic resolvent ring
of $Q$, are parametrized by $G(\Z)$-orbits on $V(\Z)$.  The definition
of the cubic resolvent is not important for this section.
\begin{theorem}\label{thquarticrings}
There is a natural bijection between the set of $G(\Z)$-equivalence
classes on $V(\Z)$ and the set of isomorphism classes of pairs
$(Q,C)$, where $Q$ is a quartic ring and $C$ is a cubic resolvent ring
of $Q$.
\end{theorem}
The congruence conditions defining maximality may be found in
\cite{B1}. The action of $G$ on $V$ has a unique polynomial invariant
$\Delta$ called the {\it discriminant}. If $(A,B)\in V(\Z)$
corresponds to the pair $(Q,C)$, then we have
$\Delta(A,B)=\Delta(Q)=\Delta(C)$.

To count $S_4$-quartic fields having prescribed splitting conditions,
we directly use a result of Ellenberg--Pierce--Wood \cite[Theorem
  4.1]{EPW}, which improves on the results of \cite{BBP}, which in
turn builds on work of Bhargava \cite{B1} determining asymptotics for
the counting function of $S_4$-quartic fields.

\begin{theorem}\label{thquarticcount}
Let $\FF$ be the parametrized family of quartic $S_4$-fields ordered
by discriminant.  Let $\epsilon>0$. Then, for any prime $p$ and
conjugacy class $\tau\in \mathcal{T}_n$, we have
\begin{equation}\label{quarticcount}
\begin{array}{rcl}
	\displaystyle
        |\FF(x)|&=&\displaystyle\frac{5\beta}{24}x+O_{\epsilon}(x^{23/24+\epsilon}),\\[.2in]
        \displaystyle
        |\FF^{p,\tau}(x)|&=&\displaystyle\frac{5c_{p,\tau}\beta}{24}x+O_{\epsilon}(x^{23/24+\epsilon}
        p^{1/2+\epsilon}),\\[.2in] \displaystyle
        |\FF^{p\mid\Delta}(x)|&=&\displaystyle\frac{5c_{p\mid\Delta}\beta}{24}x+
        O_{\epsilon}(x^{23/24+\epsilon}
        p^{1/2+\epsilon}),
\end{array}
\end{equation}
where $\beta=\displaystyle\prod_p(1+p^{-2}-p^{-3}-p^{-4})$,
$c_{p,\tau}=\displaystyle\frac{|\tau|}{24}\frac{p^3}{p^3+p^2+2p+1}$,
and $c_{p\mid\Delta}=\displaystyle\frac{(p+1)^2}{p^3+p^2+2p+1}$. 
\end{theorem}

This verifies Equations \eqref{e:thmainsncount} and \eqref{cptau} for
the parametrized family of $S_4$-fields with
$\delta_0=\delta_1=23/24+\epsilon$ and $A=B=1/2+\epsilon$, and yields the bound
$\alpha<1/24$ of the support in Theorem~\ref{th:low-lying}.

\subsection{The family of quintic fields}
Let $V=4\otimes\wedge^2(5)$ denote the space of quadruples of $5\times
5$-skew symmetric matrices. We represent elements in $V$ as
$(A,B,C,D)$. The group $G=\GL_4\times\SL_5$ acts on $V$ via the action
$$
(g_1,g_2)\cdot (A,B,C,D):=(g_2^tAg_2,g_2^tBg_2,g_2^tCg_2,g_2^tDg_2)\cdot g_1^t.
$$ 
A result of Bhargava \cite{Bquintic} states that isomorphism classes of pairs
$(Q,R)$, where $Q$ is a quintic ring and $R$ is a sextic resolvent
ring of $Q$, are parametrized by $G(\Z)$-orbits on $V(\Z)$.  Given an
element $(A,B,C,D)\in4\otimes\wedge^2(5)$, the corresponding five
points in $\P^3$ are obtained as the intersection of the five $4\times
4$-Pfaffians of $Ax+By+Cz+Dt$.

Again, the definition of a sextic resolvent ring is not important for
us. See \cite{Bquintic} for a precise description.
\begin{theorem}\label{thquinticcrings}
There is a natural bijection between the set of $G(\Z)$-equivalence
classes on $V(\Z)$ and the set of isomorphism classes of pairs
$(Q,R)$, where $Q$ is a quintic ring and $R$ is a sextic resolvent
ring of $Q$.
\end{theorem}
The congruence conditions defining maximality may be found in
\cite{B2}. The action of $G$ on $V$ has a unique polynomial invariant
$\Delta$ called the {\it discriminant}. If $(A,B,C,D)\in V(\Z)$
corresponds to the pair $(Q,R)$, then we have
$\Delta(A,B,C,D)=\Delta(Q)=\Delta(R)$.

To count $S_5$-quintic fields having prescribed splitting, we directly
use \cite[Theorem 5.1]{EPW}.

%To count quintic $S_5$-fields, we use identical techniques as in the
%previous subsection. The asymptotic count for the number of
%$G(\Z)$-orbits on $V(\Z)$ was obtained in \cite{B2}.

\begin{theorem}\label{thquinticcount}
Let $\FF$ be the parametrized family of quintic $S_5$-fields ordered
by discriminant.  Let $\epsilon>0$. Then, for any prime $p$ and
conjugacy class $\tau\in \mathcal{T}_n$, we have
\begin{equation}
\begin{array}{rcl}
\displaystyle |\FF(x)|&=&\displaystyle\frac{13\beta}{120}x+O(x^{199/200+\epsilon}),\\[.2in]
\displaystyle |\FF^{p,\tau}(x)|&=&\displaystyle\frac{13c_{p,\tau}\beta}{120}x+O(x^{199/200+\epsilon})+O(x^{79/80+\epsilon} p^{1/2+\epsilon}),\\[.2in]
\displaystyle |\FF^{p\mid\Delta}(x)|&=&\displaystyle\frac{13c_{p\mid\Delta}\beta}{120}x+O(x^{199/200+\epsilon})+O(x^{79/80+\epsilon} p^{1/2+\epsilon}),
\end{array}
\end{equation}
where $\beta=\displaystyle\prod_p(1+p^{-2}-p^{-4}-p^{-5})$, $c_{p,\tau}=\displaystyle\frac{|\tau|}{120}\frac{p^4}{p^4+p^3+2p^2+2p+1}$, and $c_{p\mid\Delta}=\displaystyle\frac{(p+1)(p^2+p+1)}{p^4+p^3+2p^2+2p+1}$.
\end{theorem}
%The values of the constants $c_{p,\tau}$ and $c_{p\mid\Delta}$ can
%be determined from \cite[Lemma 20]{Bquintic}.
%
The additional error of $O(X^{199/200+\epsilon})$ (in comparison with
the quartic case) arises from the bound on the number of quintic
orders that are not $S_5$-orders obtained in~\cite{ShTs}. Both
\cite[Theorem 5.1]{EPW} and \cite{ShTs} use the methods in \cite{B2}
used to determine asymptotics for the counting function of quintic
fields.

This verifies Equations \eqref{e:thmainsncount} and
\eqref{cptau} for the parametrized family of
$S_5$-fields, this time with $\delta_0=199/200+\epsilon$,
$\delta_1=79/80+\epsilon$ and $A=B=1/2+\epsilon$. This yields the bound
$\alpha<1/100$ of the support in Theorem~\ref{th:low-lying}.

\section{Other parametric families of quadratic and cubic fields}\label{s:other}
In this section, we consider families of quadratic and cubic fields
obtained by different parametrizations. The quadratic fields will be
constructed as quadratic resolvents of $S_3$-fields. The cubic fields
will be constructed as resolvents of $S_4$-fields. Thus this section
is an example of constructing one family from another.
We shall verify the Sato-Tate equidistribution for these families and
show that the assumptions of Theorem~\ref{th:low-lying} are satisfied
which enables us to determine that the symmetry type of the low-lying
zeros is symplectic.

\subsection{A parametric family of quadratic fields}
Every quadratic field can be written uniquely in the form
$K=\Q(\sqrt{d})$, where $d$ is a fundamental discriminant.
The discriminant of such a field $K$ is equal to $d$. Let $\zeta_K$
denote the zeta function of $K$. It factors as
$\zeta_K(s)=\zeta(s)L(s,\chi)$, where $\zeta(s)$ is the Riemann zeta
function and $L(s,\chi)$ is the Dirichlet $L$-function corresponding
to the quadratic character $\chi$ defined by the Kronecker symbol
$\chi(n)=\bigl(\frac{d}{n}\bigr)$. The conductor of this $L$-function
is equal to $|d|$.

We consider the family $\FF$ of quadratic fields arising as the
quadratic resolvents of nowhere totally ramified cubic fields. A cubic
field $K_3$ is said to be {\it nowhere totally ramified} if no prime
$p$ factors as $\mathcal{P}^3$ in $K_3$. Suppose that $K_3$ is a
nowhere totally ramified cubic $S_3$-extension of $\Q$ having
discriminant $D$. Let $K_6$ denote the Galois closure of $K_3$, and
$K$ denote the unique quadratic subfield of $K_6$. The field $K$ is
called the {\it quadratic resolvent field} of $K_3$. It follows that
$K_6$ is an unramified cubic extension of $K$ and that the
discriminant of $K$ is $D$.  Thus the family $\FF$ is parametrized as
\begin{equation*}
\FF=\{\Q(\sqrt{\Delta(f)}):f\in \GL_2(\Z)\backslash\Sym^3(\Z^2)^{\rm{ntr}}\},
\end{equation*}
where $f$ ranges over $\GL_2(\Z)$-orbits of maximal integral binary
cubic forms that are nowhere totally ramified. We order elements in
$\FF$ by discriminant. For each $x\ge 1$, the set $\FF(x)$ consists of
the quadratic fields in $\FF$ having discriminant less than $x$ in
absolute value.  Note that the quadratic fields in $\FF$ occur with
multiplicities. In fact, \cite{DH} implies that quadratic fields $K$
appear in $\FF$ with a multiplicity of $(\#\Cl(K)[3]-1)/2$, where
$\Cl(K)$ denotes the class group of $K$.  Therefore, it is also
possible to think of $\FF$ as a weighted family of quadratic fields,
where each field $K$ is weighted with $(\#\Cl(K)[3]-1)/2$. However, we
prefer to consider $\FF$ as a geometric family arising from the space
of integral binary cubic forms.

Recall the branched covering $X\to V$ of degree three (described in
the introduction), where $V\simeq \A^4$ is the space of binary cubic forms and
$X\subset V\times \P^1$ is the zero locus.  We construct $Y\subset
V\times \P^1$ defined by the zero locus of the polynomial $x^2 -
\Delta(f)y^2$. This is a branched covering $Y\to V$ of degree
two. Clearly it is $\GL_2$-equivariant. Restricting to $V^{\Delta\neq
  0}$ we obtain an \'etale covering. The \'etale covering $Y\to
V^{\Delta\neq 0}$ is obtained from the \'etale covering $X\to
V^{\Delta\neq 0}$ by the resolvent construction applied to this
relative situation (i.e.\ applied to each fiber). Our parametric
family $\FF$ is attached to the covering $Y\to V$ as in
Section~\ref{s:setup}. Alternatively, we could have constructed the
family $\FF$ starting from $X\to V$, but using the one-dimensional
Artin representation $\Gal(K_6/\Q) \to S_3\to\GL_1(\C)$ (the sign
character).
%However, since this latter viewpoint does not directly conform with the set-up in Section~\ref{s:setup} we shall not use it.

\subsection{Symmetry type corresponding to this family of quadratic fields}

For our purposes it will be necessary to relate the splitting type of
$p$ in a nowhere totally ramified cubic field $K_3$ to the splitting
type of $p$ in the quadratic resolvent of $K_3$. The splitting type of
a prime $p$ in $K_3$ determines the splitting type of $p$ in $K_6$,
the Galois closure of $K_3$, and hence determines the splitting type
of $p$ in $K_2$, the quadratic resolvent of $K_3$. These splitting
types can be immediately computed by applying the method of
\cite{Wood}, yielding the following lemma.
% The proof of the above theorem follows immediately by applying the method of \cite{Wood}.
\begin{lemma}\label{thcqsp}
Let $K_3$ be a cubic field that is nowhere totally ramified, and let $K_2$
denote its quadratic resolvent field.  If $p$ has splitting type
$(111)$ or $(3)$ in $K_3$ then $p$ has splitting type $(11)$ in $K_2$ and
if $p$ has splitting type $(21)$ in $K_3$ then $p$ has splitting type
$(2)$ in $K_2$.
\end{lemma}

The asymptotics of $|\FF(x)|$ is the result of Davenport-Heilbronn
\cite[Theorem 3]{DH} on the average $3$-part of the class group of quadratic
fields (this result is restated in \cite[Theorem 2]{BST}, and a
simpler proof is provided). The counting result \cite[Theorem
  1.4]{TaTh}, in conjunction with Lemma~\ref{thcqsp}, implies the
analogues of Equations \eqref{e:thmainsncount} and \eqref{cptau} for
$\FF$, with $\delta_0=5/6$, $\delta_1=18/23+\epsilon$, $A=20/23$, and
$B=40/23$. Thus $\FF$ is an $S_2$-family in the sense that for a fixed
prime $p$, the splitting types $(11)$ and $(2)$ occur equally often in
$\FF$.

As in Section~\ref{s:setup}, we define the average conductor $\cL$
which in fact coincides with the average conductor of the family of
cubic fields. Theorem~\ref{th:low-lying} then follows for a
Paley-Wiener function $f$ whose Fourier transform has support in
$[-\alpha,\alpha]$, with $\alpha< \frac{10}{103}$.

Since $(\#\Cl(K)[3]-1)/2$ is equal to the number of index-$3$
subgroups of $\Cl(K)$, $\FF(x)$ can be viewed as a weighted set of
$L$-functions $L(s,\chi_d)$ arising from all quadratic fields
$K=\Q(\sqrt{d})$, where each field is counted with multiplicity
$(\#\Cl(d)[3]-1)/2$.  Since the Sato-Tate measure of the unweighted
family of quadratic fields is the same, we deduce the same Sato-Tate
equidistribution also when the fields are counted with multiplicity
$\#\Cl(d)[3]$. The same holds for the symplectic symmetry type of
low-lying zeros, so we can for example deduce, when summing over
positive fundamental discriminants $d$, that
\begin{equation}\label{quadraticresult}
\lim_{x\to\infty}\frac{\pi^2}{4 x}\sum_{0< d <x}\#\Cl(d)[3]\sum_j f\bigg(\frac{\gamma^{(j)}_d\cL}{2\pi}\bigg)=\widehat{f}(0)-f(0)/2.
\end{equation}
Note that $\#\Cl(d)[3]$ is $\frac{4}{3}$ on average over
asymptotically $\frac{3x}{\pi^2}$ positive fundamental discriminants
$0<d<x$.

\subsection{A parametric family of cubic fields}

We now consider a family of cubic fields arising as cubic resolvents
of certain quartic fields. 
Given a quartic $S_4$-field
$K_4$, let $K_{24}$ denote its Galois closure. The field $K_6$,
corresponding to the subgroup $V_4\subset S_4$ generated by the double
transpositions in $S_4$, is Galois and its Galois group is
$S_4/V_4\cong S_3$. Let $K_3$ denote a cubic $S_3$-field contained in
$K_6$ ($K_3$ is unique up to conjugation). Then $K_3$ is called the
{\it cubic resolvent field} of $K_4$.

A quartic field $K_4$ is said to be {\it nowhere overramified} if no
rational prime $p$ has splitting type $(1^21^2)$, $(2^2)$, or $(1^4)$
in $K_4$. If $K_4$ is a nowhere overramified quartic field and its
cubic resolvent field is $K_3$, then the discriminant of $K_4$ is
equal to the discriminant of $K_3$. To give a description of the
family of cubic resolvents of nowhere overramified quartic fields as a
geometric family, we have the following theorem that is a result of
Bhargava \cite{Bquartic}.
\begin{theorem}\label{thmparcubicres}
Let $(Q,C)$ be a pair of rings, where $Q$ is the maximal order of a
nowhere overramified quartic field $K_4$ and $C$ is the (unique) cubic
resolvent ring of $Q$. Let $(A,B)$ be a pair of integral ternary
quadratic forms such that the $\GL_2(\Z)\times\SL_3(\Z)$-orbit of
$(A,B)$ corresponds to $(Q,C)$ under the bijection of {\rm
  \cite[Theorem~1]{Bquartic}}. Then, under the Delone-Faddeev
parametrization {\rm \cite{DF}}, the cubic ring $C$ corresponds to the
binary cubic form $4\det(Ax-By)$. Furthermore, $C$ is the maximal
order of the cubic resolvent field of $K_4$.
\end{theorem}

We now define our family $\FF$ of cubic fields as follows:
\begin{equation*}
\FF=\{K(f):f=4\det(Ax-By),\;(A,B)\in(\GL_2(\Z)\times\SL_3(\Z))\backslash(\Z^2\otimes\Sym^2(\Z^3))^{{\rm nor}}\},
\end{equation*}
where $K(f)$ denotes the cubic field that is the field of fractions of
the cubic ring corresponding to $f$, and $(A,B)$ runs over
$\GL_2(\Z)\times\SL_3(\Z)$-orbits of maximal integral pairs of ternary
quadratic forms that are nowhere overramified. Note that Theorem
\ref{thmparcubicres} implies that the discriminant of $K(f)$ is equal
to the discriminant of $(A,B)$. We order elements in $\FF$ by
discriminant and denote the set of elements in $\FF$ with discriminant
less than $x$ by $\FF(x)$.

Let $V$ denote the space of pairs of ternary quadratic forms.  Given a
generic element $(A,B)\in V$, we obtain four points in $\P^2$, namely,
the four points of intersection of the quadrics corresponding to $A$
and $B$. We also obtain three points in $\P^1$, namely, the three
roots of the cubic resolvent form $4\det(Ax-By)$ of $(A,B)$. We thus
obtain the natural space $Z\subset V\times\P^2\times\P^1$, and a
degree-12 branched covering $Z\to V$. Taking the intersection of $Z$
with $V\times\P^2$, we obtain a branched covering $X\to V$ of degree
four, and taking the intersection of $Z$ with $V\times\P^1$, we obtain
a branched covering $Y\to V$ of degree three. All three branched
coverings are $\GL_2\times\SL_3$-equivariant.

Consider the family of $L$-functions associated to $\FF$, where for
each cubic $S_3$-field $K_3\in\FF$, we take the Artin $L$-function
$L(s,\rho_{K_3})$ corresponding to the standard representation of
$S_3$. This family arises naturally from the branched covering $Y\to
V$. However, we note that we may also form this family of
$L$-functions from the branched covering $X\to V$. Indeed, for an
$S_4$-quartic field $K_4$ with Galois closure $K_{24}$, we associate
to it the two-dimensional Artin representation $\Gal(K_{24}/\Q)\cong
S_4\to S_3\to\GL_2(\C)$, where $S_4\to S_3$ is the map in which we
quotient out by the subgroup generated by double transpositions. The
corresponding family of $L$-functions is the same as the family
associated to the branched cover $Y\to V$. This is because the field
in $K_{24}$ fixed by the double transpositions of $S_4$ is the
degree-$6$ Galois closure $K_6$ of the cubic resolvent $K_3$ of
$K_4$. Hence the Artin $L$-function corresponding to the
representation $\Gal(K_{24}/\Q)\cong S_4\to S_3\to\GL_2(\C)$ is the
same as the Artin $L$-function corresponding to the standard
representation $\Gal(K_{6}/\Q)\cong S_3\to\GL_2(\C)$.

%As in Section~\ref{s:setup}, we may also form $\FF$ from the covering
%$Y\to V$.  To a cubic field $K\in\FF$, we associate the $L$-function
%$L(s,\rho_K)$, where $\rho_K$ is related to the standard
%representation of $S_3$ as in \eqref{rhoK}. Alternatively, we could
%have constructed the same family of $L$-functions by starting with the
%family of nowhere overramified quartic $S_4$-fields $K_4$ (arising
%from the covering $X\to V$), and using the two-dimensional Artin
%representation $\Gal(K_{24}/\Q)\cong S_4\to S_3\to\GL(2,\C)$, where
%$S_4\to S_3$ is the map in which we quotient out by the subgroup
%generated by double transpositions.
%Again, we do not use the latter interpretation of $L$-functions
%because it is not amenable to the setup of \S2.

As in the case of the family of quadratic resolvents of cubic fields,
cubic fields $K\in\FF$ arise with multiplicities. The following
theorem, due to Heilbronn \cite{heilbronn}, shows that the multiplicity of $K$
is $\#\Cl(K)[2]-1$:
\begin{theorem}\label{th:Heilbronn}
Let $K$ be a fixed cubic $S_3$-field. Then index-$2$ subgroups of
$\Cl(K)$ are in bijective correspondence with quartic fields that are
nowhere overramified and have $K$ as a cubic resolvent field.
\end{theorem}
Therefore, it is possible to interpret $\FF$ as a family of weighted
cubic fields, where each cubic field $K$ is weighted by
$\#\Cl(K)[2]-1$. However, as before, we prefer to consider $\FF$ as a geometric
family.

\subsection{Symmetry type corresponding to this family of cubic fields}

We will need to relate the splitting type of an unramified prime $p$
in a nowhere overramified quartic field $Q$ to the splitting type of
$p$ in the cubic resolvent field $C$ of $Q$. This is done in the
following proposition:
\begin{proposition}\label{t:quartic}
Let $Q$ be a quartic order, and let $C$ be a cubic resolvent of
$Q$. Fix a prime $p$ that does not ramify in $Q$. The splitting type
of $p$ in $Q$ determines the splitting type of $p$ in $C$. Table
\ref{table} lists the different possible pairs of splitting types.
\end{proposition}
\begin{table}[ht]
\centering
\begin{tabular}{|c | c| c| }
  \hline &&\\[-8pt] Splitting type of $p$ in $Q$ & Splitting type of
  $p$ in $C$ & Density in $\FF$
\\[12pt] \hline &&\\[-8pt] $(1111)$&
  $(111)$& $\displaystyle\frac1{24}+O\Bigl(\frac{1}{p}\Bigr)$\\[12pt]
  $(22)$& $(111)$&
  $\displaystyle\frac1{8}+O\Bigl(\frac{1}{p}\Bigr)$\\[12pt] $(211)$&
  $(21)$& $\displaystyle\frac1{4}+O\Bigl(\frac{1}{p}\Bigr)$\\[12pt]
  $(4)$& $(21)$&
  $\displaystyle\frac1{4}+O\Bigl(\frac{1}{p}\Bigr)$\\[12pt] $(31)$&
  $(3)$& $\displaystyle\frac1{3}+O\Bigl(\frac{1}{p}\Bigr)$\\[12pt]
\hline
\end{tabular}
\caption{Densities of splitting types}\label{table}
\end{table}
\begin{proof}
Let $K_4$ and $K_3$ denote $Q\otimes_\Z\Q$ and $C\otimes_\Z \Q$,
respectively. Since $K_4$ is a quartic field and $Q$ is an order in $K_4$, we deduce that $K_3$ is also a field and $C$ is an order in $K_3$.  Since $p$ is unramified in $K_4$, it remains unramified in the Galois closure of $K_4$, and hence in $K_3$. The splitting types of $p$ in $K_4$ and $K_3$ are
the same as the splitting types of $p$ in $Q$ and $C$,
respectively. The theorem now follows by applying the method of
\cite{Wood}.
\end{proof}

The counting results of Theorem \ref{thquarticcount} imply the
analogue of \eqref{e:thmainsncount} with
$\delta_0=\delta_1=23/24+\epsilon$, $A=12$, and $B=11$.  The analogue
of \eqref{cptau} follows from Table \ref{table}
because $1/24+1/8=1/6$, $1/4+1/4=1/2$, and $1/3=1/3$. Thus $\FF$ is an
$S_3$-family in the sense that the Frobenius elements are uniformly
distributed in $\T_3$.

As in Section~\ref{s:setup}, if we define $\cL$ to be the average
conductor, then Theorem \ref{th:low-lying} follows for a Paley-Wiener
function $f$ whose Fourier transform has support in
$[-\alpha,\alpha]$, with $\alpha< \frac{1}{24}$.  Since
$\#\Cl(K)[2]-1$ is equal to the number of index-$2$ subgroups of
$\Cl(K)$, $\FF(x)$ can be viewed as a weighted set of Artin
$L$-functions arising from $S_3$-fields $K$, where each $S_3$-field is
counted with multiplicity $\#\Cl(K)[2]-1$. Therefore, in conjunction
with Section~\ref{s:families}, we deduce equidistribution results for
cubic fields counted with multiplicity $\#\Cl(K)[2]$.

The main ingredient that we use in order to consider these weighted
families (of quadratic fields $K$ weighted by $\#\Cl(K)[3]$ and of
cubic fields $L$ weighted by $\#\Cl(L)[2]$) is that these weighted
families can be parametrized in terms of integral orbits of reductive
groups on certain representations.  Let us also note that it is
possible to obtain analogous results for the families of quadratic and
cubic fields weighted by $\# \mathrm{Cl}(K)\cdot {\rm Reg}(K)$.  One
way to obtain such a result is to use geometric families that
parametrize quadratic and cubic fields, with these weights.  For
quadratic fields, we use the space of binary quadratic forms
modulo the $\SL_2$-action, and for cubic fields, we use the space
$\Z^2\otimes\Z^3\otimes\Z^3$ modulo the $\GL_2\times \SL_3\times
\SL_3$-action (see \cite{Bquintic}). Though integral orbits on these
representations parametrize simply the class groups of quadratic
(resp.\ cubic) fields, the fundamental domains that can be most
naturally constructed for these spaces weigh each quadratic
(resp.\ cubic) field $K$ by $\# \mathrm{Cl}(K)\cdot {\rm Reg}(K)$. The
latter construction can be seen in \cite{SiegelBQF} for the case of
quadratic fields and \cite[Chapter 2]{Wilson} for cubic fields.
Alternatively, we can use the fact that the Dirichlet class number
formula expresses this quantity as a residue at $s=1$ of $\zeta_K(s)$
which can be well approximated by a short Dirichlet polynomial.  On
the other hand arithmetic weights such as $L(\tfrac12,\chi_d)$ could
change the answer.

%For quadratic fields, we use the the space of binary quadratic forms
%modulo the $\SL_2$-action, and for cubic fields, we use the space
%$\Z^2\otimes\Z^3\otimes\Z^3$ modulo the $\GL_2\times \SL_3\times
%\SL_3$-action (see \cite{Bquintic}).

\section{$S_n$-families}\label{s:n-ary}
In this section, we consider the parametric family of monogenized
degree-$n$ number fields and prove Theorem~\ref{th:nic} concerning the
Sato-Tate equidistribution.  We will let $V\simeq \A^n$ denote the
space of monic polynomials of degree $n$.
The ring of functions of $V$ is $\Z[a_1,\ldots,a_n]$, which we can
identify via the fundamental theorem of algebra with
$\Z[x_1,\ldots,x_n]^{S_n}$.  Thus we can identify $V$ with the GIT
quotient $\A^n//S_n$, which simply amounts to factor the monic
polynomial
\[
f(T) = T^n + a_1 T^{n-1} +\cdots a_n = (T-x_1)(T-x_2)\cdots (T-x_n).
\]
Equivalently $V$ is the Hilbert scheme of $n$ points in $\A^1$.  The
ring of functions of the cartesian product $V\times \A^1$ is
$\Z[x_1,\ldots,x_n]^{S_n}[T]$.  The subscheme $X\subset V\times \A^1$
corresponding to the zero set of $f$ is defined by the principal ideal
generated by $(T-x_1)(T-x_2)\cdots (T-x_n)$.
\begin{lemma}
	There is a ring isomorphism between the quotient ring $\Z[x_1,\ldots,x_n]^{S_n}[T]/\langle(T-x_1)(T-x_2)\cdots (T-x_n)\rangle$, and $\Z[x_1,\ldots,x_n]^{S_{n-1}}$, induced by specializing $T\mapsto x_n$.
\end{lemma}
\begin{proof}
The map is clearly a ring homomorphism. It is not difficult to verify
that it is injective.  Since the image contains the polynomial $x_n$,
to prove surjectivity it suffices to show that the image contains the
subring $\Z[x_1,\ldots,x_{n-1}]^{S_{n-1}}$. Let
$g(x_1,\ldots,x_{n-1})$ be an $S_{n-1}$-invariant polynomial. Then
$g(x_1,\ldots,x_{n-1}) = f(x_1,\ldots,x_{n-1},0)$ for some
$S_n$-invariant polynomial $f$. This can be proved by using the
elementary symmetric polynomials, see
e.g.~\cite[\S1.1]{Macdonald:sym-orth}. Then $g$ is the image of
$f(x_1,\ldots,x_n-T)\in \Z[x_1,\ldots,x_n]^{S_n}[T]$.
\end{proof}

From this lemma, we deduce that the ring of functions of $X$ can be
identified with $\Z[x_1,\ldots,x_n]^{S_{n-1}}$. Equivalently $X$ is
the Hilbert scheme of $n$ points in $\A^1$, one of which is marked.
Similarly the Galois closure $\widetilde X$ of $X\to V$ is identified
with $\A^n$, which parametrizes $n$ marked points in $\A^1$, and with
ring of functions $\Z[x_1,\ldots,x_n]$.

In this section, the family
$\FF$ consists of the degree-$n$ fields corresponding to $\Z$-orbits
on $V(\Z)^\max$, the set of elements $f$ in $V(\Z)$ such that
$\Z[x]/f(x)$ is a maximal order in a degree-$n$ field.

\subsection{Monogenized fields arising from monic integer polynomials}
Recall from the introduction the notion of monogenized rings and
fields.
%We say that the pair
%$(R,\alpha)$ is a {\it monogenized} ring. Two monogenized rings
%$(R,\alpha)$ and $(R',\alpha')$ are said to be isomorphic if $R$ is
%isomorphic to $R'$ and $\alpha$ is carried to $\alpha'+m$ under this
%isomorphism for some integer $m\in\Z$.  A pair $(K,\alpha)$ is said to
%be a {\it monogenized degree-$n$ field} if $K$ is a degree-$n$ field
%and $\alpha\in\cO_K$ is such that $\cO_K=\Z[\alpha]$, where $\cO_K$ is
%the ring of integers of $K$. Two monogenized fields $(K,\alpha)$ and
%$(K',\alpha')$ are said to be isomorphic if the monogenized rings
%$(\cO_K,\alpha)$ and $(\cO_{K'},\alpha')$ are isomorphic.
%Let $V(\Z)$ denote the space of monic integer polynomials of degree$n$.
A polynomial $f(T)\in V(\Z)$ gives rise to the monogenized ring
$(\Z[T]/(f(T)),T)$. Conversely, a monogenized ring $(R,\alpha)$, where
$R$ has rank $n$ over $\Z$, gives rise to a polynomial $f\in V(\Z)$,
namely, the characteristic polynomial of $\alpha$. The group $\Z$
acts on $V(\Z)$ via the action $(m\cdot f)(T)=f(T+m)$. Since the
characteristic polynomial of $\alpha+m$ is $f(T-m)$, where $f$ is the
characteristic polynomial of $\alpha$, it follows that the isomorphism
classes of monogenized rank-$n$ rings are in bijection with the
$\Z$-orbits on $V(\Z)$.

%Let $R$ be a rank-$n$ ring. An element $\alpha\in R$ is said to be a
%{\it monogenizer} for $R$ if $R=\Z[\alpha]$. The pair $(R,\alpha)$ is
%then said to be a {\it monogenized ring}.

%Let $(R,\alpha)$ be a monogenized rank-$n$ ring. Let $f(T)\in\Z[T]$ be
%the minimal polynomial of $\alpha$. Then we have
%$R\cong\Z[T]/f(T)$. Furthermore, if $\alpha'=\alpha+m$ is a different
%monogenizer for $R$, then 

%Let $$f(T)=T^n+a_1T^{n-1}+\cdots+a_n$$ be a non-zero element of
%$V(\Z)$.  We let $K_f$ denote the finite $\Q$-algebra $\Q[T]/f(T)$ and
%let $R_f$ denote the ring $\Z[T]/f(T)$.  The discriminant of $R_f$ is
%equal to $\Delta(f)$, the discriminant of $f$.  It is possible that
%different $\Z$-orbits $f$ in $V(\Z)$ give rise to the same ring
%$R_f$. However, it follows from a result of Birch and Merriman
%\cite{BM} that each nondegenerate rank-$n$ ring arises only from
%finitely many $\Z$-orbits.

In this section, we shall consider the family $\FF$ of degree-$n$
fields $K_f$ that arise as the fraction fields of \emph{maximal}
orders $R_f$ corresponding to $\Z$-orbits of integer monic degree-$n$
polynomials. These fields are said to be {\it monogenic}. This family
is distinct from the family of fields arising from \emph{all} orders
corresponding to integer monic degree-$n$ polynomials (see
\cite{LagariasWeiss}). The latter family would capture all
$S_n$-fields since every number field is generated by a single element
over $\Q$ and thus every number field is the field of fractions of
some (possibly non-maximal) order corresponding to an integer monic
degree-$n$ polynomial. Moreover, every degree-$n$ field arises in the
latter family infinitely often.

It is expected that for $n\geq 3$, most maximal orders (in fact,
most rings) are not monogenic. Thus, we expect that our family of
monogenic fields is thin in the full set of degree-$n$ fields, though
this is not known to be the case for any $n\geq 3$.  
For example, a cubic ring corresponding to the binary cubic form
$f$ under the Delone--Faddeev correspondence \cite{DF} is
monogenic if and only if $f$ represents $1$ over $\Z$.
So in the case $n=3$, the thinness of the family of monogenic
cubic rings reduces to the open question of showing that $100\%$ of
integral binary cubic forms do not represent $1$.

%The Galois group of a generic degree-$n$ monogenic field is $S_n$. This fact,
%i.e.\ that a $100\%$ of monogenic fields have Galois group $S_n$, follows
%from Hilbert's irreducibility theorem.  A power saving error term follow with
%an application of the Selberg sieve.

In the next subsection, we consider the family $\FF$ of monogenized
fields and define an appropriate height function on it.  We then
determine asymptotics for the number of monogenized fields having
prescribed splitting conditions at a fixed prime $p$, and use these
asymptotics to determine the symmetry type of the low-lying zeros of
the corresponding family of $L$-functions.

\subsection{Counting results}
Every $\Z$-orbit on $V(\Z)$ has a unique representative whose
$T^{n-1}$-coefficient is between $0$ and $n-1$. Let $V(\Z)_k$ denote
the set of monic integer polynomials whose $T^{n-1}$-coefficient is
$k$. Then the set $\bigsqcup_{k=0}^{n-1}V(\Z)_k$ is a set of orbit
representatives for the action of $\Z$ on $V(\Z)$. 
% Decompose $\FF=\bigsqcup^{n-1}_{k=0} \FF_k$ where $\FF_k$ denote the set of degree-$n$ fields $K$ arising from irreducible elements in $V(\Z)_k$.

%For each prime $p$ let $\Sigma_p$ denote a set of isomorphism
%classes of degree-$n$ \'etale algebras over $\Q_p$. We say that the
%collection $\Sigma=(\Sigma_p)_p)$ is {\it acceptable} if for large
%enough primes $p$, the set $\Sigma_p$ contains all the algebras whose
%discriminants are not divisible by $p^2$. 
%such that $K\otimes\Q_p$ is isomorphic to
%$\Sigma_p$ for all primes $p$.

Let $f(T)=T^n+a_1T^{n-1}+\cdots+a_n$ be an element of $V(\Z)$. The
coefficients $a_i$ are the $i$th symmetric polynomials evaluated on
the roots of $f$. Hence we consider $a_i$ to be a degree-$i$ function
on $V$. The discriminant $\Delta$ is a degree $n(n-1)$ function on
$V$.  We then define the following height on $V(\R)$:
\begin{equation*}
h(T^n+a_1T^{n-1}+a_2T^{n-2}+\cdots+a_n)=\max_i\big\{|a_i|^{n(n-1)/i}\big\}.
\end{equation*}
% For $x>0$, let $V(\R)_{k,h<x}$ (resp.\ $V(\Z)_{k,h<x}$) denote the set of
For $x>0$, we consider the set of elements in $V(\R)_k$ (resp.\ $V(\Z)_k$) having height less than $x$. Then
\begin{equation*}
\begin{array}{rcl}
|\{f\in V(\Z)_k:\ h(f)< x \}|
&=&\Vol( \{f\in V(\R)_k:\ h(f)< x \})+O\big(x^{\frac{n+2}{2n}-\frac{2}{n(n-1)}}\big)\\[.2in]
&=&2^{n-1}x^{\frac{n+2}{2n}}+O\big(x^{\frac{n+2}{2n}-\frac{2}{n(n-1)}}\big).
\end{array}
\end{equation*}
See also~\cite{Deng} for the related count of rational points on weighted projective spaces.

An application of Hilbert irreducibility proves that $100\%$ of
elements $f$ in $V(\Z)$ yield $\Q$-algebras $K_f=\Q[T]/f(T)$ which are
$S_n$-number fields. 
%We now consider the elements $f$ in $V(\Z)_{k}$ such that $K_f$ is not an $S_n$-number field. 
Indeed, an application of the Selberg sieve yields the
upper bound \footnote{This bound is essentially due to Gallagher. See
  the recent article~\cite{Dietmann} for more general results via a
  different approach based on resolvent rings and the method of
  Bombieri-Pila for counting integer points on high degree curves.}
\begin{equation}\label{eqnsnmonic}
|\{
f\in V(\Z)_k:\ K_f \text{ not $S_n$-field},\ h(f)< x
\}|
=
O_{\epsilon}\big(x^{\frac{n+2}{2n}-\frac{2}{5n(n+1)}+\epsilon}\big).
\end{equation}

%Let $\FF$ denote the set of monogenized $S_n$-number fields having
%squarefree discriminant.
%Fields in $\FF$ are automatically
%$S_n$-number fields. Indeed, the Galois group $G$ of the Galois
%closure of a field $K$ having squarefree discriminant over $\Q$ is
%generated by the inertia subgroups inside $G$. However, since $K$ has
%squarefree discriminant, each inertia subgroup must be isomorphic to
%$\Z/2\Z$. And the only transitive subgroup of $S_n$ generated by
%transpositions is $S_n$.

Next, we consider the subsets $V(\F_p)^{(\tau)}\subset V(\F_p)$ of
polynomials having splitting type $\tau\in\T_n$ and the subset
$V(\F_p)^{\Delta=0}\subset V(\F_p)$ of polynomials that have
discriminant $0$. As in \S3, we denote the set of elements in
$V(\Z_p)$ corresponding to maximal degree-$n$ extensions of $\Z_p$ by
$V(\Z_p)^\max$. We also let $V(\Z_p)^{p\mid\Delta}$ denote the set of
elements in $V(\Z_p)$ whose discriminants are divisible by $p$, and
let $V(\Z_p)^{p\mid\Delta,\max}$ denote $V(\Z_p)^\max\cap
V(\Z_p)^{p\mid\Delta}$.
\begin{lemma}\label{l:langweil}
Let $p$ be a prime such that $(p,n)=1$. Then we have
\begin{equation}
\begin{array}{rcccl}
\displaystyle
c_{p,\tau}&:=&\displaystyle\frac{|V(\F_p)^{(\tau)}|}{p^n}\cdot
\frac{1}{\Vol(V(\Z_p)^\max)}&=&
\displaystyle\frac{|\tau|}{|S_n|}+O\Bigl(\frac1{p}\Bigr),\\[.2in]
c_{p\mid\Delta}&:=&\displaystyle\frac{|V(\F_p)^{\Delta=0}|}{p^n}\cdot
\frac{\Vol(V(\Z_p)^{p\mid\Delta,\max})}
     {\Vol(V(\Z_p)^{p\mid\Delta})\Vol(V(\Z_p)^\max)}
     &=&O\displaystyle\Bigl(\frac1{p}\Bigr).
\end{array}
\end{equation}
\end{lemma}
\noindent As mentioned in~\cite[\S2.2]{SST} the lemma follows from the
analogue of the Chebotarev equidistribution for \'etale coverings which
can be established with the Lang-Weil bound. Below we give an
elementary proof.

\vspace{5pt}

\begin{proof}
A set $S$ of $n$ points in $\A^1(\overline{\F}_p)$ is said to be {\it
  defined over $\F_p$} if the set $S$ is fixed by the Galois group of
$\overline{\F}_p$ over $\F_p$. A monic degree-$n$ polynomial with
coefficients in $\F_p$ yields a set of $n$ points in
$\A^1(\overline{\F}_p)$ defined over $\F_p$, namely its
roots. Conversely, given a set of $n$ points in
$\A^1(\overline{\F}_p)$ defined over $\F_p$, it determines a unique
monic degree-$n$ polynomial with coefficients in $\F_p$.

The number of sets of $n$ points in $\A^1(\overline{\F}_p)$ defined
over $\F_p$ is $p^n$. If $\Delta(f)=0$ for $f\in V(\F_p)$, then
the corresponding set of $n$ points contains at least one point
counted with multiplicity greater than 1. The number of such sets is
$\sim p^{n-1}$ which proves the second part of the lemma.

Now consider an unramified splitting type
$\tau=(n)^{n_n}\ldots(2)^{n_2}(1)^{n_1}$, where $\sum j n_j=n$. If
$f\in V(\F_p)$ has splitting type $\tau$, then the corresponding set
of $n$ points consists of $n_1$ distinct points in $\A^1(\F_p)$, $n_2$
distinct pairs of conjugate points in
$\A^1(\F_{p^2})\backslash\A^1(\F_p)$, and so on. Up to an error term
of $O(p^{n_1-1})$, the number of sets of $n_1$ distinct points in
$\A^1(\F_p)$ is $p^{n_1}/n_1!$. Similarly, the number of sets of $n_k$
distinct $k$-tuples of conjugate points in
$$\A^1(\F_{p^k})\backslash\big(\bigcup_{\substack{d\mid k\\d\neq
    k}}\A^1(\F_{p^d})\big)$$ is $p^{kn_k}/(k^{n_k}\cdot
n_k!)+O(p^{kn_k-1})$, since the number of $k$-tuples of conjugate
points in $\A^1(\F_{p^k})$ is $p^{k}/k$.  Thus, the number of sets of
$n$ points in $\A^1(\overline{\F}_p)$ defined over $\F_p$
corresponding to the splitting type $\tau$ is equal to
\begin{equation}\label{eqlem511}
\frac{p^{n_1}}{n_1!}\frac{p^{2n_2}}{2^{n_2}\cdot
  n_2!}\cdots\frac{p^{kn_k}}{k^{n_k}\cdot n_k!}\cdots\frac{p^{n n_n}}{n^{n_n}\cdot
  n_n!}+O(p^{n-1})=\frac{p^n}{|\Stab_{S_n}(\tau)|}+O(p^{n-1}),
\end{equation}
where the equality follows since the cardinality of the stabilizer of $\tau$ in $S_n$ is
exactly equal to the denominator of the main term in the left-hand
side of the above equation.

Next, we note that conditions of maximality for $f\in V(\Z_p)$ are
listed in \cite[Corollary 3.2]{ABZ}. In particular, if $f\in V(\Z)$ is
nonmaximal, then either the reduction of $f$ modulo $p$ has a double
root $\alpha\in\F_p$ such that $p^2\mid f(\tilde{\alpha})$ for any
lift $\tilde{\alpha}\in\Z_p$ of $\alpha$, or $f$ has multiple repeated
roots. In either case it follows that $p^2\mid\Delta(f)$ for
nonmaximal elements $f\in V(\Z_p)$. Hence, we see that the volume of
the set of nonmaximal elements is bounded by $O(1/p^2)$. Therefore, we
obtain
\begin{equation}\label{eqlem512}
  \begin{array}{ccl}
    \displaystyle\frac{1}{\Vol(V(\Z_p)^\max)}&=&
    1+O\displaystyle\Bigl(\frac{1}{p^2}\Bigr),\\[.2in]
\displaystyle\frac{\Vol(V(\Z_p)^{p\mid\Delta,\max})}
     {\Vol(V(\Z_p)^{p\mid\Delta})\Vol(V(\Z_p)^\max)}
     &=&1+O\displaystyle\Bigl(\frac{1}{p}\Bigr).
  \end{array}
\end{equation}
The lemma follows from \eqref{eqlem511}, \eqref{eqlem512}, and the
orbit-stabilizer formula which gives
$|\Stab_{S_n}(\tau)||\tau|=|S_n|$.
\end{proof}

%The following lemma follows immediately from Lemma \ref{l:langweil}
%\begin{lemma}\label{l:langweil}
%We have
%\begin{equation}
%\begin{array}{rcccl}
%\displaystyle
%c_{p,\tau}&:=&|V_k(\F_p)^{(\tau)}|&=&
%\displaystyle\frac{|\tau|}{|S_n|}p^{n-1}+O(p^{n-2}),\\[.25in]
%\displaystyle c_{p\mid\Delta}&:=&|V_k(\F_p)^{\Delta=0}|&=&O\displaystyle(p^{n-2}).
%\end{array}
%\end{equation}
%\end{lemma}

\begin{lemma}\label{lemABZ}
For any prime $p$, we have
$$
\rho(p):=\frac{\#V(\Z/p^2\Z)^\max}{\#V(\Z/p^2\Z)}=1-\frac{1}{p^2}.
$$
\end{lemma}
\begin{proof}
This is \cite[Proposition 3.5]{ABZ} combined with~\cite[Corollary 3.2]{ABZ}.
\end{proof}

The asymptotics for the number of elements in $V(\Z)$ having bounded
height and squarefree discriminant is computed in \cite{BSW}. The key
ingredient in that result is the following ``tail estimate'' proved
in \cite[Theorem 1.5]{BSW}:
\begin{equation}\label{equnifmonic1}
\sum_{\substack{m>M\\\mu^2(m)=1}} |\{f\in V(\Z):\ m^2|\Delta(f),\ h(f)<x
\}|=O_\epsilon\big(x^{\frac{n+1}{2n-2}+\epsilon}/M\big)+O_\epsilon
\big(x^{\frac{n+1}{2n-2}-\frac{1}{5n(n-1)}+\epsilon}\big).
\end{equation}
% where for $m\geq 1$, $\W_{m}(x)$ denotes the set of elements in $V(\Z)$
% having height bounded by $x^{1/n(n-1)}$ and discriminant divisible by $m^2$. 
% Note that the height in \cite{BSW} is slightly different from the height we use: namely the height is
% $$
% H(T^n+a_1T^{n-1}+\cdots+a_n)=\max_i\{|a_i|^{1/i}\}
% $$

Recall that $\FF$ denotes the family of $\Z$-orbits in $V(\Z)^{\max}$. Thus for $x\ge 1$, $\FF(x)$ is in bijection with the set of monogenized fields in $\FF$ arising from irreducible integer monic polynomials having height bounded by $x$.
Let $\FF^{p,\tau}(x)$ and $\FF^{p\mid\Delta}(x)$ correspond to the set of
fields $K$ in $\FF(x)$ such that the splitting type of $p$ in $K$ is
$\tau$ and such that $p\mid\Delta(K)$, respectively.  
Using arguments identical to those in \cite{BSW}, we
estimate the number of elements in $\FF(x)$, $\FF^{p,\tau}(x)$,
and $\FF^{p\mid\Delta}(x)$.
\begin{theorem}\label{thmonnmt}
Let $c_{p,\tau}$ and $ c_{p\mid\Delta}$ be as in Lemma
\ref{l:langweil}. We have
\begin{equation}\label{eqpolymaincount}
\begin{array}{rcl}
  \displaystyle |\FF(x)|&=&\displaystyle \frac{2^{n-1}n}{\zeta(2)}x^{\frac{n+2}{2n}}+O_\epsilon\big(x^{\frac{n+2}{2n}-\frac1{5n(n-1)}+\epsilon}\big),\\[.2in]
	\displaystyle
        |\FF^{p,\tau}(x)|&=&\displaystyle c_{p,\tau}|\FF(x)|+O_{\epsilon}\big(x^{\frac{n+2}{2n}-\frac1{5n(n-1)}+\epsilon}
        p^{n}\big),\\[.2in] \displaystyle
        |\FF^{p\mid\Delta}(x)|&=&\displaystyle
        c_{p\mid\Delta}|\FF(x)|+O_{\epsilon}\big(x^{\frac{n+2}{2n}-\frac1{5n(n-1)}+\epsilon}
        p^{n-1}\big).
\end{array}
\end{equation}
%\footnote{NT: Not clear that the exponents of $p$ are right. Need to double-check that }
\end{theorem}
\begin{proof}
We start by computing the number of elements $f\in V(\Z)_k$ having
height less than $x$. The coefficients $a_i$ of such an $f$ are as
follows: $a_1=k$ and $|a_i|< x^{i/(n(n-1))}$ for $2\leq i\leq n$. Thus
there are a total of $\sim 2^{n-1}x^{\delta}$ such elements $f$, where
$\delta=(n(n+1)/2-1)/(n(n-1))=(n+2)/(2n)$.

For squarefree positive $m$, let $\U_{k,m}(x)$ denote the set of
elements $f\in V(\Z)_k$ such that $R_f$ is nonmaximal at every prime
dividing $m$ and $h(f)<x$.  Let $\U_m(x)$ denote
$\bigsqcup_{k=0}^{n-1}\U_{k,m}(x)$.  Note that if $R_f$ is nonmaximal
at a prime $p$, then $p^2\mid\Delta(f)$.  Since the discriminant of
$f(T)$ is equal to the discriminant of $f(T+a)$ for integers $a$,
\eqref{equnifmonic1} immediately implies the following estimate:
\begin{equation}\label{equnifmonic}
\sum_{\substack{m>M\\\mu^2(m)=1}} |\U_{m}(x)|=O_\epsilon\big(x^{\frac{n+2}{2n}+\epsilon}/M\big)+O_\epsilon\big(x^{\frac{n+2}{2n}-\frac{1}{5n(n-1)}+\epsilon}\big).
\end{equation}
(We exclude the factors of $n$ in the error terms since $n$ is assumed to be fixed.)

The set $\U_{k,m}$ is defined via congruence conditions modulo $m^2$.
Let $V(\Z/m^2\Z)_k^\nmax$ denote the set of elements whose lifts to
$V(\Z)_k$ are nonmaximal at every prime dividing $m$. Then for
\begin{equation*}
\rho_k(m):=\frac{\#V(\Z/m^2\Z)_k^\nmax}{\#V(\Z/m^2\Z)_k},
\end{equation*}
we have
\begin{equation*}
|\U_{k,m}(x)|=\rho_k(m)2^{n-1}x^{\frac{n+2}{2n}}+O\big(x^{\frac{n+2}{2n}-\frac{2}{n(n-1)}}\big).
\end{equation*}

%We have the following lemma:
%\begin{lemma}

Since the condition of maximality (and hence of nonmaximality) is
$\Z$-invariant, it follows that the density of nonmaximal elements in
$\bigsqcup_{k=0}^{n-1} V(\Z/m^2\Z)_k$ is equal to the density of
nonmaximal elements in $V(\Z/m^2\Z)$.  Furthermore, the size of
$V(\Z/m^2\Z)_k$ is equal to $m^{2n-2}$ independent of $k$. It
therefore follows that the average of $\rho_k(m)$ is equal to the
density of nonmaximal elements in $\bigsqcup_{k=0}^{n-1}
V(\Z/m^2\Z)_k$, and can be computed from Lemma \ref{lemABZ} using
the multiplicativity over $m$ of this density:
%The average of $\rho_k(m)$ over $k\in\{0,\ldots,n-1\}$ is thus equal to 
%$$
%\frac{1}{nm^{2n-2}}\sum_{k}\#V(\Z/m^2\Z)_k^\nmax=\frac{\#V(\Z/m^2\Z)^\nmax}{nm^{2n-2}},
%$$
%By multiplicativity, it follows from Lemma \ref{lemABZ} that for
%squarefree $m$, the average value of $\rho_k(m)$ over all
%$k\in\{0,\ldots,n-1\}$ is equal to
\begin{equation*}\frac{1}{n}\sum_{k=0}^{n-1}\rho_k(m)=\prod_{p\mid m}1/p^2=1/m^2.
\end{equation*}
Therefore, from \eqref{eqnsnmonic} and \eqref{equnifmonic}, we have for $\delta>0$,
\begin{equation*}
\begin{array}{rcl}
|\FF(x)|&=&\displaystyle\sum_{m\geq 1}\mu(m) |\U_m(x)|+O_{\epsilon}\big(x^{\frac{n+2}{2n}-\frac{2}{5n(n+1)}+\epsilon}\big)\\[.2in]
&=&\displaystyle\sum_{k=0}^{n-1}\sum_{m\geq 1}\mu(m)|\U_{k,m}(x)|
+O_{\epsilon}\big(x^{\frac{n+2}{2n}-\frac{2}{5n(n+1)}+\epsilon}\big)\\[.2in]
&=&\displaystyle\sum_{k=0}^{n-1}\sum_{m=1}^{x^{\delta}}\mu(m)\rho_k(m)2^{n-1}x^{\frac{n+2}{2n}}+O\Bigl(\sum_{m=1}^{x^{\delta}}x^{\frac{n+2}{2n}-\frac{2}{n(n-1)}}\Bigr)
+O_\epsilon\Bigl(x^{\frac{n+2}{2n}-\delta+\epsilon}+x^{\frac{n+2}{2n}-\frac1{5n(n-1)}+\epsilon}\Bigr)\\[.2in]
&=&\displaystyle \frac{2^{n-1}n}{\zeta(2)}x^{\frac{n+2}{2n}}+O_\epsilon\Bigl(
x^{\frac{n+2}{2n}-\delta}+x^{\frac{n+2}{2n}-\frac{2}{n(n-1)}+\delta}+
x^{\frac{n+2}{2n}-\delta+\epsilon}+x^{\frac{n+2}{2n}-\frac1{5n(n-1)}+\epsilon}\Bigr).
\end{array}
\end{equation*}
We pick $\delta=1/n(n-1)$ and obtain the first estimate of the
theorem. The proof of the other two estimates are identical. We simply
count points\footnote{Here we have chosen not to optimize the exponent
  of $p$ in the error terms; using the methods in \cite{EPW}
  would yield significantly improved error bounds.} in the translates of $pV(\Z)$
corresponding to Lemma \ref{l:langweil}.
\end{proof}

%\begin{proof} The first equality in \eqref{eqpolymaincount} was
%proved in \cite[]{}. The key ingredient in that proof was the
%following ``tail estimate'' in \cite[]{}: $$
%\sum_{\substack{m>M\\\mu^2(m)=1}}\#\U_{m,x}=O_\epsilon(x^{\frac{n+1}{2n-2}+\epsilon}/\sqrt{M})+O_\epsilon(x^{\frac{n+1}{2n-2}-\frac1{(5n(n-1)}+\epsilon}), $$
%where $\U_{m,x}$ denotes the set of monic integer polynomials $f$
%such that $h(f)<x$ and $m^2\mid\Delta(f)$.  The second and third
%equalities of \eqref{eqpolymaincount} follow from the above tail
%estimate just as in the proof of Theorem
%\ref{thquarticcount}.  \end{proof}

This concludes the proof of the Sato-Tate equidistribution for this
family of monogenized degree-$n$ fields. Therefore, by the results of
\S2 in conjunction with Lemma \ref{l:langweil} and Theorem
\ref{thmonnmt}, we see that the symmetry type of the family is
symplectic and that the bound on the support given by Theorem
\ref{th:low-lying} is $\alpha<2/(5n(n-1)(2n+1))$.

\subsection{Fields arising from binary $n$-ic forms }
Let $W=\Sym^n(2)$ denote the space of all binary $n$-ic forms.  A
construction of Nakagawa \cite{Nakagawa} attaches a degree-$n$ ring
$R_f$ to a nondegenerate integral binary $n$-ic form $f$.  The
following geometric construction of $R_f$ is due to Wood \cite[Theorem
  2.4]{Woodbnf}: to an integral binary $n$-ic form $f\in W(\Z)$, we
associate its scheme $X_f$ of zeros and the ring $R_f$ of regular
functions on $X_f$.  This produces a quasiprojective scheme $X\subset
W\times\P^1$ which is also a branched covering $X\to W$ of degree
$n$. We may consider the family of fields arising from integral binary
$n$-ic forms that correspond to maximal orders in $S_n$-fields. This
yields a family of $L$-functions as before. The Sato-Tate
equidistribution for this family would follow in identical fashion
from a tail estimate, analogous to~\eqref{equnifmonic1} but for binary
$n$-ic forms $f\in W(\Z)$ such that $m^2 | \Delta(f)$.

\section{Mixed families}\label{s:mixed}

In this section, we consider geometric families that are {\it mixed},
i.e., the fields yielding the $L$-functions in the families do not all
have the same Galois group. However, each family we consider can be
naturally partitioned into disjoint subfamilies where the Galois group
is constant. The Sato-Tate group of the subfamily is then equal to
this Galois group embedded into the torus of some $\GL_n$. The
Sato-Tate group of the mixed family is just the subgroup of the torus
generated by the Sato-Tate groups of all the subfamilies. We thus
verify part (ii) of Conjecture 1
in~\cite{SST}.

\subsection{Binary cubic forms and $S_2$-, $C_3$- and 
$S_3$-fields}\label{mixed cubic}

Let $\FF$ denote the family of \'etale cubic extensions of $\Q$
arising as $R\otimes \Q$, where $R$ corresponds to $\GL_2(\Z)$-orbits
on the set of maximal integral binary cubic forms that have nonzero
discriminant and do not factor as the product of three linear forms
over $\Q$.  (We omit the cases of $\Q\oplus \Q\oplus \Q$ which
corresponds to binary cubic forms which split completely over $\Q$.)
We have a disjoint decomposition $\FF = \FF_C \sqcup \FF_S \sqcup
\FF_Z$, where the family $\FF_C$ corresponds to $C_3$-fields, the
family $\FF_S$ corresponds to $S_3$-fields (and has been studied in
Section~\ref{s:families}), and the family $\FF_Z$ corresponds to the
direct sums $K\oplus \Q$ of quadratic fields $K$ and $\Q$.

The Sato-Tate group of $\FF_S$ is $S_3\subset\GL_2(\C)$ embedded via
the standard representation of $S_3$. For $\FF_C$ the Sato-Tate group
is $C_3\subset S_3 \subset \GL_2(\C)$, and for $\FF_Z$ the Sato-Tate
group is $S_2\subset S_3\subset \GL_2(\C)$. 
%Let $T$ denote the torus of $\GL_2(\C)$ consisting of diagonal matrices, and let $W$ denote the Weyl group.  
Recall that $\TT = (S^1)^2/S_2$ consists of pairs of unit complex numbers modulo
permutation of the two coordinates.
We shall use the notation $\delta(a,b)$, where $a,b\in S^1$,
to denote the Dirac delta measure supported at the point $(a,b)\in \TT$. Let
$\rho$ denote a nontrivial cube root of unity, and let $\bar{\rho}$
denote the complex conjugate of $\rho$. Then the Sato-Tate measures
and the indicators $i_1$, $i_2$, and $i_3$ for the families $\FF_C$,
$\FF_S$, and $\FF_Z$ are listed in Table \ref{tableSTcubic}.

\begin{table}[hbt]\label{t:mixed3}
\centering
\begin{tabular}{|c | c|c|c|c|}
  \hline &\\[-12pt]
\multirow{2}{*}{Family}&\multirow{2}{*}{Sato-Tate measure} &\multicolumn{3}{c|}{Indicators}\\ \cline{3-5}
&&$i_1$&$i_2$&$i_3$\\
\hline
  $\FF_C$&$\frac{1}{3}\delta(1,1)+\frac{2}{3}\delta(\rho,\overline{\rho})$
&2&2&0
\\[.2in]
$\FF_Z$&$\frac{1}{2}\delta(1,1)+\frac{1}{2}\delta(1,-1)$&2&2&2
\\[.2in]
$\FF_S$&$\frac{1}{6}\delta(1,1)+\frac{1}{2}\delta(1,-1)+\frac{1}{3}\delta(\rho,\overline{\rho})$&1&1&1
\\
\hline
\end{tabular}
\caption{Sato-Tate measures and the corresponding indicators for the families in \S\ref{mixed cubic}.}\label{tableSTcubic}
\end{table}

Note that $S_2$ and $C_3$ do not act irreducibly on $\C^2$; as a
consequence $i_1(\FF_Z)=i_1(\FF_C)=2$. This is apparent since the
$L$-function attached to $K\oplus \Q\in \FF_Z$ is
$\zeta_K(s)=\zeta(s)L(s,\chi)$, where $\chi$ is a Dirichlet
character. Similarly, the $L$-function corresponding to a cyclic cubic
field in $\FF_C$ is a product of two Dirichlet $L$-functions
$L(s,\chi)L(s,\overline{\chi})$, where $\chi$ is a cubic
character. The family of cubic character $L$-functions $L(s,\chi)$ is itself of
unitary symmetry type since it has Sato-Tate group $C_3\subset \GL_1(\C)$. It
has been studied for example in~\cite{David:cubic}.

It is a consequence of the work of Davenport--Heilbronn that the
families $\FF_S$ and $\FF_Z$ occur with positive proportion $c_S$ and
$c_Z$ inside $\FF$, while $\FF_C$ has zero proportion.
By construction the Sato-Tate measure of the family
$\FF$ is
\[
\mu_{\rm ST}(\FF) = c_S \cdot \mu_{\rm ST}(\FF_S) + c_Z \cdot \mu_{\rm ST}(\FF_Z).
\]
The indicators are easily calculated and in fact Table~\ref{tableSTcubic}
shows that all three are equal to $c_S+2c_Z$. Similarly, the
statistics of the low-lying zeros of $\FF$ are simply the
superposition of those of the family $\FF_S$ of $S_3$-fields and the
family $\FF_Z$ of $S_2$-fields, weighted with the respective
proportions $c_S$ and $c_Z$.

\subsection{Pairs of ternary quadratic forms and $S_4$- and $D_4$-fields}\label{sub:D4}

Recall that $\GL_2(\Z)\times\SL_3(\Z)$-orbits on the space of pairs of
integral ternary quadratic forms correspond bijectively to isomorphism
classes of pairs $(Q,R)$, where $Q$ is a quartic ring and $R$ is a
cubic resolvent ring of $Q$. We may use this parametrization to
construct a mixed family $\FF$ of $S_4$ and $D_4$ quartic fields.
Given a pair $(A,B)$ of integral ternary quadratic forms given in
Gram-matrix form and corresponding to a pair of rings $(Q,R)$ as
above, the cubic resolvent form is defined to be
$f(x,y)=4\det(Ax-By)$. Then, under the Delone--Faddeev correspondence
\cite{DF}, the integral binary cubic form $f$ corresponds to the cubic
resolvent ring $R$ of $Q$ (see \cite{Bquartic}). In particular, if $Q$
is an $S_4$-, $D_4$-, or $A_4$-ring, then $R$ is an $S_3$-ring, an
order contained in a direct sum of $\Q$ and a quadratic field, or a
$C_3$-ring, respectively.

When quartic fields are ordered by discriminant, $D_4$-fields occur
with a positive proportion. This is related to the shape of a
fundamental domain for the $\GL_2(\Z)\times \SL_3(\Z)$-action on
$V(\R)$, and the presence of ``cusps''. More precisely, one of the
cusps of this fundamental domain contains only integral elements
$(A,B)$ with $\det(A)=0$ (see Case II in the proof of \cite[Lemma
  11]{B1}). This implies that if the corresponding quartic ring is
nondegenerate, then it is either a $D_4$-ring or an order contained
within the direct sum of two quadratic fields. A $100\%$ of
$D_4$-rings are contained in this cusp, and they make up a positive
proportion of irreducible quartic rings.

To construct a geometric family of quartic fields that are either
$S_4$ or $D_4$, we must restrict to pairs of integral ternary
quadratic forms that have nonzero discriminant and are
maximal. Furthermore, in order to exclude $A_4$-rings and reducible
rings, we impose the condition that the prime 2 stays inert. That is,
we consider the family of fields $\FF$ arising as the field of
fractions of rings corresponding to the set of
$\GL_2(\Z)\times\SL_3(\Z)$-orbits on maximal nondegenerate elements of
$V(\Z)$ whose splitting type at 2 is $(4)$.  We then have the
decomposition $\FF = \FF_S \sqcup \FF_D$, where $\FF_S$ consists of
$S_4$-fields and $\FF_D$ consists of $D_4$-fields. This is one
instance where one can perform rigorously the decomposition alluded to
in the remarks concerning assertion (ii) of Conjecture~1
in~\cite{SST}.

Quartic $D_4$-orders may also be distinguished from $S_4$-orders using
the property that a cubic resolvent of a $D_4$-order is a suborder of
$\Q\oplus K$, where $K$ is a quadratic field. In contrast, the cubic
resolvent of an $S_4$-order is an order in an $S_3$-cubic
field. Exploiting this difference, Wood \cite[Section 7.3]{WoodThesis}
parametrizes quartic rings with reducible resolvent rings as
$G$-orbits, where $G$ is a subgroup of $\GL_2(\Z)\times \GL_3(\Z)$, on
the space of pairs $(A,B)$ of integral ternary quadratic forms such
that $A$ has top row zero and $a_{12}\neq 0$. The quartic rings that
arise include all $D_4$-orders but no $S_4$-orders.  Let $U$ be the
space of quadruples $(A,B,x,y)$ of two integral ternary quadratic
forms $A$ and $B$ as above and two integers $x$ and $y$ such that
$\det(Ax - By)=0$. Geometrically the cusp containing $D_4$-fields
arises from the natural equivariant map $U \to V$. Indeed, the common
zero locus of $A$ and $B$ in $\P^2$ yields branched $4$-covers of $V$
and $U$ whose normal closures have Galois group $S_4$ and $D_4$
respectively. The family $\FF_{D}$ is parametrized by the
$\GL_2(\Z)\times \SL_3(\Z)$ orbits in $U(\Z)$.

The Sato-Tate group for $\FF_S$ is $S_4\subset\GL_3(\C)$ embedded via
its standard representation, and the Sato-Tate group for $\FF_D$ is
$D_4\subset S_4\subset\GL_3(\C)$.  
%Let $T$ denote the torus of $\GL_3(\C)$ consisting of diagonal matrices, and let $W$ denote the Weyl group. 
We identify $\TT$ with $(S^1)^3$ modulo
permutation of the three coordinates and use the notation
$\delta(a,b,c)$ to denote the Dirac delta measure supported at the
point $(a,b,c)$. Then the Sato-Tate measures and the indicators $i_1$,
$i_2$, and $i_3$ for the families $\FF_S$ and $\FF_D$ are listed in
Table~\ref{tableSTquartic}.  The computations of the indicators are
elementary, and only require the character of the standard
representation of $S_4$. The computations of the Sato-Tate measures are
more involved, but follow from the proofs of Baily's \cite{Baily} and Bhargava's
\cite{B1} counting results on $D_4$- and $S_4$-fields, respectively.

\begin{table}[hbt]\label{t:mixed4}
\centering
\begin{tabular}{|c | c|c|c|c|}
  \hline &\\[-12pt]
\multirow{2}{*}{Family}&\multirow{2}{*}{Sato-Tate measure} &\multicolumn{3}{c|}{Indicators}\\ \cline{3-5}
&&$i_1$&$i_2$&$i_3$\\[.05in]
\hline
  $\FF_S$&$\frac{1}{24}\delta(1,1,1)+\frac{1}{4}\delta(1,1,-1)+
\frac{1}{3}\delta(1,\rho,\overline{\rho})+\frac{1}{4}\delta(-1,i,-i)+
\frac{1}{8}\delta(1,-1,-1)$
&1&1&1
\\[.2in]
$\FF_D$&$\frac{1}{8}\delta(1,1,1)+\frac{1}{4}\delta(1,1,-1)
+\frac{1}{4}\delta(-1,i,-i)+\frac{3}{8}\delta(1,-1,-1)$&2&2&2
\\
\hline
\end{tabular}
\caption{Sato-Tate measures and the corresponding indicators for the families in \S6.2.}\label{tableSTquartic}
\end{table}

Since $D_4$ does not act irreducibly on $\C^3$, the family $\FF_D$ is not essentially cuspidal. In fact $\C^3$ decomposes as the direct sum of the standard $2$-dimensional representation of $D_4$ and the character of $D_4$ whose kernel is generated by the two transpositions of $D_4$, so we can write $\FF_D = \FF_{\rm dih} \oplus \FF_{\rm char}$. Geometrically we can describe $\FF_{\rm char}$ as a branched $2$-cover constructed as follows: for $(A,B,x,y)\in U$, the zero set of the degenerate ternary quadratic form $Ax - By$ is the union of two lines in $\mathbb{P}^2$.  These two lines are joining opposite points of the common zero locus of $A$ and $B$ and the interpretation is that this picture is preserved by the action of $D_4$.

The quantitative Sato-Tate equidistribution for $\FF_{\rm
  dih}$ and $\FF_{\rm char}$ follows from the methods in \cite{Baily}
and \cite{B1}, respectively. If we further assume a power saving error
estimate for the count of $D_4$-fields, then, since the dihedral
representation $D_4\subset \GL_2(\C)$ and the above character
$D_4\to \GL_1(\C)$ are orthogonal, we expect both $\FF_{\rm dih}$ and
$\FF_{\rm char}$ to have symplectic symmetry type. The distribution of
the low-lying zeros for the family $\FF_D$ will then be the independent
direct sum of two $\Sp(\infty)$ ensembles. In particular, for
restricted support of the level densities, we expect there to be no correlation
between the zeros of $L(s,\mathrm{dih}_K)$ and $L(s,\mathrm{char}_K)$
as $K$ ranges over $D_4$-fields.

As in \S6.1, we find the Sato-Tate measure of the family $\FF$ to be
\[
\mu_{\rm ST}(\FF) = c_S \cdot \mu_{\rm ST}(\FF_S) + c_D \cdot \mu_{\rm ST}(\FF_D),
\]
where $c_S$ and $c_D$ denote the proportions of $\FF_S$ and $\FF_D$ inside $\FF$. It follows immediately from Table~\ref{tableSTquartic} that the corresponding indicators satisfy $i_1(\FF)=i_2(\FF)=i_3(\FF)=c_S+2c_D$. Finally, the
statistics of the low-lying zeros of $\FF$ are simply the
superposition of those of the families $\FF_S$ and $\FF_D$, weighted with the respective
proportions $c_S$ and $c_D$.

\section{Local equidistribution for $S_n$-families}\label{s:mass}

 %various $S_n$-families, is concerned with the splitting behavior of
 %unramified primes. 
 Sections \ref{s:families} and \ref{s:n-ary} were concerned with the splitting behavior of unramified primes in families of number fields. In the present section we investigate the ramified
 primes. We conclude the section with a reformulation of Bhargava's heuristics~\cite{B3} for counting number fields, via a comparison with Peyre's constant~\cite{Peyre} for the counting of rational points on Fano varieties.

 \subsection{Binary $n$-ic forms}\label{s:loceqnic}
We begin with the family $\FF=\FF_{\nic}$ of binary $n$-ic forms.  Each field in this family can be thought of as
arising from a set $S=X_f$ of $n$ points in $\P^1(\overline{\Q})$
defined over $\Q$. The absolute Galois group of $\Q$ acts on this set,
and the fixed subgroup of the Galois group cuts out the number field
$M=M_f$ corresponding to these $n$ points.  Thus, we obtain an
injection $$\Gal(M/\Q)\hookrightarrow \Aut(S)\cong S_n.$$ Fix a prime
$p$. We may consider the reduction $\overline{S}=X_f\otimes_\Z \F_p$
of $S$ modulo $p$ which yields $n$ points, counted with multiplicity,
in $\P^1(\overline{\F_p})$. This set $\overline{S}$ is defined over
$\F_p$, and the absolute Galois group of $\F_p$ acts on it. The
splitting type of $p$ is determined by the orbit decomposition of this
action. More precisely, if $\overline{S}$ breaks up into orbits having
size $a_1,a_2,\ldots,a_\ell$ (in decreasing order), then the splitting type of $p$ in $K=K_f$ is
$\sigma=\sigma(\overline{S})=(a_1a_2\ldots
a_\ell)$. In particular, $p$ is unramified in $K$ if and only
if $\overline{S}$ consists of $n$ distinct points. Motivated by this,
we define the splitting type of such a set $\overline{S}$ to be
$\sigma\in \mathcal{ST}_n$, where $\mathcal{ST}_n$ denotes the set of
all possible splitting types on $n$ points.

Consider the (finite) set of sets of $n$ points, counted with
multiplicity, in $\P^1(\overline{\F_p})$ that are defined over
$\F_p$. We have seen in the proof of Lemma~\ref{l:langweil} that it is
naturally identified with the set of non-zero forms $V(\F_p) - \{0\}$
modulo multiplication by $\GL_1(\F_p)=\F_p^\times$.
Thus for every prime $p$ we have a map $V(\F_p)-\{0\} \to
\mathcal{ST}_n$. Pushing forward the counting measure on $
V(\Z/p^2\Z)^{p^{2}\nmid \Delta}$ via
\begin{equation}\label{VZp2}
V(\Z/p^2\Z)^{p^{2}\nmid \Delta} \rightarrow V(\F_p) - \{0\}
\twoheadrightarrow \mathcal{ST}_n,
\end{equation} 
we obtain a measure $\mu_{\nic,p}$ on
$\mathcal{ST}_n$.

If $\FF$ is the family of fields arising from integral binary $n$-ic
forms having squarefree discriminant, then each field $K_f\in\FF$,
corresponding to $f$, arises from the set $X_f$ of the $n$ roots of
$f$ in $\P^1(\overline{\Q})$. We have natural reduction maps for every
prime $p$,
\begin{equation*}
\begin{array}{rcl}
	\FF(x)&\to& \mathcal{ST}_n \\
K_f&\mapsto&\sigma(X_f\otimes_\Z \F_p).
\end{array}
\end{equation*}
We expect that the image of $\FF(x)$
gets equidistributed in $V(\Z/p^2\Z)^{p^2\nmid \Delta}$ and therefore
also in $\mathcal{ST}_n$, as $x\to \infty$, with respect to the
measure $\mu_{\nic,p}$.

This statement would generalize~\eqref{cptau} which
is the unramified case. The unramified splitting types belong to the
subset $\T_n \subset \mathcal{ST}_n$ defined in
Section~\ref{s:setup}. The restriction of the measure
$\mu_{\nic,p}$ to $\T_n$ takes the form of the pushforward
of the normalized counting measure via the map
\[
V(\F_p)^{\Delta\neq 0} \to \mathcal{T}_n \subset \mathcal{ST}_n,
\]
which one may compare with~\eqref{VZp2}. The natural normalization of the measure is $\mu_{\nic,p}(\mathcal{T}_n)=1$.
% Finally, Lemma~\ref{l:langweil} says that the measures $\mu^{\mathcal{ST}}_{n,p}$ on $\mathcal{ST}_n$ converge as $p\to \infty$ to the Sato-Tate measure which is supported on $\T_n$.

Some of the analogous constructions are also valid for fields arising
from $n$ points in $\P^k$. However, in this case it is necessary to
count ramified points along with the additional data of local tangent
directions.

\subsection{$S_n$-fields of bounded discriminant}\label{s:rambd}
We now consider the families $\FF=\FF_\univ$ of $S_n$-fields of bounded
discriminants with $n=3,4,5$ as in Section~\ref{s:families}. In
addition to Theorems~\ref{thquarticcount} and \ref{thquinticcount},
Bhargava also established the equidistribution at ramified primes.

Let $\mathcal{ET}_{n,p}$ denote the (finite) set of degree-$n$ \'etale
extensions of $\Q_p$, and let $\mu_{\univ,p}$ be the
measure where each \'etale extension $K_p$ is weighted
proportionally to $\Disc_p(K_p)^{-1}\#\Aut(K_p)^{-1}$, and normalized by $\mu_{\univ,p}(\mathcal{T}_n)=1$. The total sum
of these relative proportions is computed in \cite[Proposition
  2.3]{B3} extending a mass formula for totally ramified extensions due to Serre. It is remarkable that
the answer is independent of the prime $p$,
including the primes $p$ dividing $n!$ which may wildly ramify.
\begin{proposition}[\cite{B3}]\label{p:etext}
For any prime $p$ and integers $0\le k \le n-1$ with $n\ge
1$, $$\sum_{\substack{[K_p:\Q_p]=n\\\Disc_p(K_p)=p^k}}\frac{1}{\#\Aut(K_p)}
=q(k,n-k),$$ where the sum in the left-hand side is over \'etale
extensions of $\Q_p$ of discriminant $p^k$ and $q(k,n-k)$ denotes the
number of partitions of $k$ into at most $n-k$ parts.
\end{proposition}

We have a natural map from $\FF(x)$ to $\mathcal{ET}_{n,p}$ which
sends $K$ to $K\otimes\Q_p$.  For $n=3, 4, 5$, the respective works of
Davenport--Heilbronn \cite{DH}, and Bhargava \cite{B1} and \cite{B2}
show that, for a fixed local \'etale degree-$n$ extension $K_p$ of
$\Q_p$, the relative proportion of $S_n$-fields $K$ that satisfy
$K\otimes\Q_p\equiv K_p$ is $\Disc_p(K_p)^{-1}\#\Aut(K_p)^{-1}$. Thus,
as we range over fields in $\FF(x)$ and let $x\to \infty$, the images
in $\mathcal{ET}_{n,p}$ are equidistributed with respect to this measure $\mu_{\univ,p}$.

The proof involves, among other things, the relation between the
measure $\mu_{\univ,p}$ and the local counting of orbits.
From the prehomogeneous vector spaces $(G,V)$ we have a surjective map
\[
 V(\Z/p^\nu\Z)^{\Delta\neq 0} \twoheadrightarrow \mathcal{ET}_{n,p},
\]
where $\nu\in \Z_{\ge 1}$ is an absolute constant to be chosen large
enough. (The existence of this $\nu$ is a concrete manifestation in
this context of Grothendieck's base change theorem). The measure
$\mu_{\univ,p}$ is equal to the pushforward of the normalized counting
measure which is established by a local density
calculation~\cite{Bquartic,Bquintic} exploiting the $G(\Z/p^\nu
\Z)$-action.

It is interesting to compare the present situation with the one in
\S\ref{s:loceqnic}. We relate the underlying sets as follows: for
each choice of $n$ and $p$, there is a natural surjective map
\[
\mathcal{ET}_{n,p}\twoheadrightarrow \mathcal{ST}_n \supset \mathcal{T}_n,
\] 
where we associate to each degree-$n$ \'etale extension $K_p$ its
splitting type. 
  This is a local counterpart of~\eqref{rhoK}, by choosing an embedding $\Gal(\Q_p) \hookrightarrow \Gal(\Q)$.
  The size of the fibers to $\mathcal{ST}_n$ can be read from the orbits of inertia as explained in~\cite{Wood,Kedlaya:mass}.
The pushforward of the measure
$\mu_{\univ,p}$ does \emph{not} coincide with the measure $\mu_{\nic,p}$, although in the limit as $p\to \infty$
both measures converge to the Haar
measure on $\mathcal{T}_n$. In fact, the restriction to $\T_n$ of
  the pushforward of $\mu_{\univ,p}$ coincides up to a
  scalar with the Haar measure on $\T_n$ because the orbit-stabilizer
  formula shows that $|\tau||\Stab_{S_n}(\tau)|=|S_n|$.

We note that to understand the
equidistribution in $\mathcal{ST}_n$, it is in general sufficient to
consider $V(\Z/p^\nu \Z)^{\Delta\neq 0}$ with $\nu=2$, while to
understand the equidistribution in $\T_n$ it is sufficient to take
$\nu=1$.

For general $n$, let $\FF_\univ(x)$ denote the set of $S_n$-number fields
having discriminant bounded by $x$.  Bhargava formulates in~\cite{B3}
conjectures for the asymptotics of $|\FF_\univ(x)|$ and for the proportion
of fields with prescribed splitting at a prime $p$.
For $n\ge 6$ the conjectures remain open. Note also that in these cases we cannot view
$\FF_\univ(x)$ as a geometric family in the sense of~\cite{SST}
because the underlying parameter space is a set of integral points of
a complicated algebraic variety which is not rational.

\subsection{An analogue of Peyre's constant}
%There is a natural inclusion of $\mathcal{ET}_{n,p}$ in the tempered spectrum of $\GL_{n-1}(\Q_p)$. Indeed, the Langlands parameter is obtained as a local counterpart of~\eqref{rhoK} and the projection to $\mathcal{ST}_n$ can be read from the size of the orbits of inertia as explained in~\cite{Wood,Kedlaya:mass}. Similarly there is a natural inclusion $\T_n \hookrightarrow \mathbb{T}$, where the later is identified with the unramified tempered spectrum of $\GL_{n-1}(\Q_p)$. 
%On the other hand, the space $\mathcal{ST}_n$ embeds in the space of Bernstein components attached to $\GL_{n-1}(\Q_p)$.

Consider the $S_n$-family $\FF=\FF_\univ$ of $n$ points in $\P^{n-2}$ for
$n=3,4,5$ as in \S\ref{s:rambd}. The measure $\mu_{\univ,p}$ on $\mathcal{ET}_{n,p}$ is normalized by $\mu_{\univ,p}(\mathcal{T}_n)=1$. 
%The pushforward of the measure $\mu^{\univ}_{n,p}$ to the unitary dual of $\GL_{n-1}(\Q_p)$.
It is established in \cite{B3} that the proportion of $S_n$-number fields that are unramified at $p$ is the inverse of
\[
|\mu_{\univ,p}| := \mu_{\univ,p}(\mathcal{ET}_{n,p}) = \sum_{k=0}^{n-1} q(k,n-k)p^{-k}.
\]
As $p\to \infty$, this is $1 + \frac1p + 
O\big(\frac{1}{p^2}\big)$.
By Proposition~\ref{p:etext}, the local density~\cite{B3} of number fields of
bounded discriminants can be written as 
\begin{equation}\label{dpff}
d_p(\FF_\univ) := |\mu_{\univ,p}|\cdot \zeta_p(1)^{-1} = \sum\limits_{k=0}^n \frac{q(k,n-k) - q(k-1,n-k+1)}{p^{k}},
\end{equation}
where $\zeta_p(1)=(1-\frac1p)^{-1}$ is the local factor of the Riemann zeta function.
Similarly there is a local density at infinity
$d_\infty(\FF_\univ)$ which is equal to the proportion of $2$-torsion elements in $S_n$. 
The theorems of Davenport--Heilbronn and Bhargava says that the number $|\FF_\univ(x)|$ of $S_n$-fields having absolute discriminant at most $x$
is asymptotic to $c_n x$ as $x\to \infty$, where the constant $c_n= \frac12 \cdot d_\infty(\FF_\univ) \cdot \prod_p d_p(\FF_\univ)$.

This may be compared with Peyre's constant attached to Fano varieties, as follows. The product of local densities is equal to the regularized product of local mass of the measures $\mu_{\univ,p}$,
\[
\prod\nolimits_p d_p(\FF_{\univ}) =
\prod\nolimits_p^* 
|\mu_{\univ,p}|
:=
\mathrm{Res}_{s=1} \zeta(s) \cdot \prod_p |\mu_{\univ,p}| \cdot \zeta_p(1)^{-1}.
\]
The regularization by $\zeta$ is explained by the prehomogeneous vector space $(G,V)$, having a ring of invariant functions generated by $\Delta$, thus the GIT quotient is rational.
The measures $\mu_{\univ,p}$ are normalized globally by the condition $\mu_{\univ,p}(\mathcal{T}_n)=1$, which can be seen as the analogue of the global normalization of Tamagawa measures.
The ``Picard rank" is interpreted to be $1$, which is also the order of pole of $\zeta(s)$.
In comparison with Peyre, the ``Tamagawa number" is $1$, and Peyre's constant ``$\alpha$" is interpreted to be $\frac{1}{2}$, which reflects the global constraint that, according to a classical result of Hasse, the discriminant of an $S_n$-field is always $\equiv 0,1 \pmod{4}$.  

The same reasoning can be applied to the general geometric families $\FF$ described in~\cite{SST}, and also to the automorphic families, to produce a conjectural leading term of the asymptotic count of $|\FF(x)|$ as $x\to \infty$.
%The formula~\eqref{dpff} raises the question whether the heuristics of~\cite{B3} can be applied for the general families 

\section{One-parameter families of quaternionic fields}\label{s:quaternion}

In this section, we study families of quaternionic number fields, i.e., Galois number fields $K$ such that $\Gal(K/\Q)$ is isomorphic to the quaternion group $Q$. The group $Q$ is a non-abelian group of order $8$ with the presentation
$$
\langle i,j,k\,|\,i^2=j^2=k^2=ijk\rangle.
$$
We denote the element $i^2=j^2=k^2=ijk$ by $-1$.

The group $Q$ has five irreducible representations. Apart from the trivial representation, $Q$ has three $1$-dimensional irreducible representations, coming from the maps which send one of $i$, $j$, and $k$ to $1$ and the other two to $-1$. We denote these three nontrivial characters by $\chi_1$, $\chi_2$, and $\chi_3$, respectively. Finally, apart from these four representations, $Q$ has an irreducible $2$-dimensional representation which we denote by $\rho$ and whose character we denote by $\chi$. The character $\chi$ sends $\pm1$ to $\pm2$ and the other elements of $Q$ to $0$. It is easy to check that the Frobenius--Schur indicator of $\chi$ is $-1$. In other words, $\rho$ is a quaternionic representation.

It follows that the zeta function $\zeta_K(s)$ factors into the product
$$
\zeta_K(s)=\zeta(s)L(s,\chi_1)L(s,\chi_2)L(s,\chi_3)L(s,\rho_K)^2,
$$
where
\[
\rho_K:\Gal(\Q)\to\Gal(K/\Q)\simeq Q \to \SL_2(\C) \subset \GL_{2}(\C).
\]
Thus in comparison with Section~\ref{s:setup}, the Artin $L$-function
of interest, i.e.\ $L(s,\rho_K)$, is constructed in a slightly
different way. The isomorphism $\Gal(K/\Q)\simeq Q$ is not unique but
$\rho_K$ is nevertheless uniquely defined up to conjugation because
$Q$ has exactly one irreducible $2$-dimensional representation. We
also note that $L(s,\rho_K)$ can be realized as the $L$-function of a
Hecke character of order four of a real quadratic extension of $\Q$,
since the irreducible $2$-dimensional representation of $Q$ is induced
from a character on either of its cyclic order-4 subgroups. Therefore
$L(s,\rho_K)$ is entire. Furthermore, the functions $L(s,\chi_i)$,
$i\in\{1,2,3\}$, are Dirichlet $L$-functions corresponding to the
three quadratic subfields of the unique biquadratic subfield $M$ of
$K$, which is also the subfield of elements of $K$ fixed by the center
$\{\pm1\}\subset Q$. For a treatment of all the above facts
regarding Artin $L$-functions, see \cite[Chapter 8]{CF}.

We now examine the question of how a rational prime $p$ splits in $K$
and $M$. In this section we will use exponents outside parentheses as
a convenient shorthand for repetitions in splitting types; we will
e.g.\ write $(1)^8$ and $(1^2)^4$ instead of $(11111111)$ and
$(1^21^21^21^2)$, respectively. In the case when $p$ is unramified,
the splitting is determined by the image of $\Frob_p$ in
$\Gal(K/\Q)$. If $\Frob_p$ has image $1$ in $Q$, then the splitting
type of $p$ is $(1)^8$ in $K$ and $(1)^4$ in $M$, since the size of
the decomposition group is $1$. Similarly, if $\Frob_p$ has image
$-1$, then the splitting type of $p$ is $(2)^4$ in $K$ and $(1)^4$ in
$M$. Otherwise, the splitting type of $p$ is $(4)^2$ in $K$ and
$(2)^2$ in $M$. If $p$ is tamely ramified, then we have the following
possibilities for the pair $(D,I)$ of the decomposition and inertia
groups: $(C_4,C_4)$, $(C_4,C_2)$, and $(C_2,C_2)$.  The corresponding
splitting types of $p$ are $(1^4)^2$, $(2^2)^2$, and $(1^2)^4$,
respectively. %In each of the three cases, it is easy to check that
%the local Artin conductor at $p$ is $2$. We shall avoid the case of
%wild ramification by always assuming that $2$ is unramified in $K$.

\begin{theorem}\label{quaternionthm}
 Let $a,b\in \Q^\times$ be such that none of $a,b,ab$ is a perfect square. Then the following assertions are equivalent:
\begin{itemize}
\item[(i)] There are quaternionic extensions of $\Q$ containing $\Q(\sqrt{a},\sqrt{b})$.

\item[(ii)] For each prime $p$ the relation $(-a,-b)_p=(-1,-1)_p$  of Hilbert symbols holds.\footnote{See \cite[\S III.5]{Neukirch} for the definition and basic properties of the Hilbert symbol.}

\item[(iii)] The quaternion algebra $\big(\frac{-a,-b}{\Q}\big)$ is isomorphic to the Hamilton quaternions $\big(\frac{-1,-1}{\Q}\big)$.

\item[(iv)] The ternary quadratic form $<a,b,ab>$ is equivalent to
  $<1,1,1>$ over $\Q$.\footnote{Here we are denoting the ternary
  quadratic form $\alpha x^2+\beta y^2+\gamma z^2$ by
  $<\alpha,\beta,\gamma>$.}

\item[(v)] There exist $\alpha,\beta,\gamma,\lambda,\mu,\nu\in\Q$ such that
\begin{equation*}
\begin{array}{rcl}
a &=& \alpha^2+\beta^2+\gamma^2\\
b &=& \lambda^2+\mu^2+\nu^2\\
0 &=& \alpha\lambda+\beta\mu+\gamma\nu.
\end{array}
\end{equation*}
\end{itemize}

\noindent Moreover, if the above assertions are satisfied, and we define
\begin{equation}\label{theta}
\theta := 1+\frac{\alpha}{\sqrt{a}}+\frac{\mu}{\sqrt{b}}+\frac{\alpha\mu-\beta\lambda}{\sqrt{ab}},
\end{equation}
then the quaternionic extensions of $\Q$ containing $\Q(\sqrt{a},\sqrt{b})$ are exactly those of the form $\Q(\sqrt{q\theta})$ with $q\in\Q^\times$.
\end{theorem}

\begin{proof}
The equivalence (ii) $\Leftrightarrow$ (iv) follows from the theory of
quadratic forms (from, for example, Chapter~6, Theorem 1.2 and Chapter
4, Theorem 1.1 in \cite{CRQF}). The equivalence (iii) $\Leftrightarrow$ (iv) follows from \cite[Proposition 2.5 (page 57)]{Lam} and the equivalence (ii) $\Leftrightarrow$ (iii) follows from class field theory, see \cite[Corollaire 1.2 (page 32)]{Vigneras}.
The assertion (v) is equivalent to
$M^TM= \left(\begin{smallmatrix} a & 0 \\ 0 &
  b \end{smallmatrix}\right)$ where $M:=\left(
	\begin{smallmatrix}
	\alpha & \lambda  \\
	\beta & \mu \\
	\gamma & \nu \\
         \end{smallmatrix}
	\right).$
Thus (iv) $\Rightarrow$ (v) is immediate. Conversely, (v)
$\Rightarrow$ (iv) follows by completing $M$ with the third
column
$$\left(\beta\nu-\gamma\mu,\gamma\lambda-\alpha\nu,\alpha\mu-\beta\lambda\right)^T$$ 
to obtain the equivalence of the quadratic forms. 

The equivalence (i) $\Leftrightarrow$ (iii) $\Leftrightarrow$ (iv) is Witt's theorem \cite{Witt}.  We outline Witt's original proof of the equivalence (i) $\Leftrightarrow$ (iii) because it doesn't seem to be well-known. Witt first considers the non-split exact sequence 
\[
 1 \to \{\pm 1\}  \to Q \to \Z/2\Z \times \Z/2\Z \to 1,
\]
where $\{\pm 1\}$ is the center and $\Z/2\Z \times \Z/2\Z$ is given by the image of $\{1,i,j,k\}$. This yields a class $\xi\in H^2(\Q(\sqrt{a},\sqrt{b}),\overline{\Q}^\times)$ of order two. Moreover $\xi$ is trivial after inflating to $\Q$ if and only if $(i)$ holds. Identifying Galois cohomology with the Brauer group, it can be verified that the inflation of $\xi$ to $\Q$ corresponds to the central simple algebra
\[
\Big(\frac{-a,-b}{\Q}\Big) \otimes \Big(\frac{-1,-1}{\Q}\Big)
\]
of dimension $16$ over $\Q$. By the properties of the Brauer group
this algebra is split over $\Q$ if and only if
$\big(\frac{-a,-b}{\Q}\big)$ is isomorphic to
$\big(\frac{-1,-1}{\Q}\big)$, which concludes the proof of (i)
$\Leftrightarrow$ (iii).

The final assertion, given the equivalent conditions (i)--(v), follows
from \cite[Theorem 4]{Kiming}. 
Let $\theta$ be given by~\eqref{theta}. Using (v), we find that the norm
of $\theta$ in $\Q(\sqrt{a})$ is equal to $\frac{1}{b}(\nu +
\frac{\alpha\nu - \gamma\lambda}{\sqrt{a}})^2$. Hence $b$ is a sum of
two squares in $\Q(\sqrt{a})$ and furthermore $\Q(\sqrt{\theta})$ is a
cyclic extension of $\Q(\sqrt{a})$. Similarly we find that
$\Q(\sqrt{\theta})$ is a cyclic extension of $\Q(\sqrt{b})$ and a
cyclic extension of $\Q(\sqrt{ab})$. Since $Q$ is the only group of
order $8$ that contains three distinct cyclic subgroups of order $4$
we deduce that $\Q(\sqrt{\theta})$ has Galois group $Q$.

Suppose that $K$ is another quaternionic extension of $\Q$ containing
$\Q(\sqrt{a},\sqrt{b})$. The composition of $K$ with
$\Q(\sqrt{\theta})$ is an extension $K(\sqrt{\theta})$ which is
biquadratic over $\Q(\sqrt{a},\sqrt{b})$ and has degree $16$ over $\Q$
with Galois group $Q\times \Z/2\Z$. The third intermediate extension
inside $K(\sqrt{\theta})$ has Galois group $(\Z/2\Z)^3$ and is of the
form $\Q(\sqrt{a},\sqrt{b},\sqrt{c})$ for some $c\in \Q^\times$. Then
we have $K(\sqrt{\theta})=\Q(\sqrt{\theta},\sqrt{c})$ and it is not
difficult to verify that $K=\Q(\sqrt{q\theta})$ for some $q\in
\Q^\times$ which may differ from $c$ above primes dividing $ab$.
\end{proof}

\begin{remark}
Let us point out that the last two paragraphs of the proof of Theorem
\ref{quaternionthm} independently establish the implication (v)
$\Rightarrow$ (i); see \cite{JensenYui,Reichardt} for variations of
this argument. We also note that \cite{Reichardt} shows how to
establish the implication (i) $\Rightarrow$ (ii) in a case by case
verification.
\end{remark}

\begin{corollary}\label{three squares}
	A necessary condition for the existence of a quaternionic extension of $\Q$ containing $\Q(\sqrt{a},\sqrt{b})$ is that each of $a,b,ab$ is a sum of three squares.
\end{corollary}

\begin{proof}
	This follows immediately from the implication (i) $\Rightarrow$ (iv) of Theorem \ref{quaternionthm}.
\end{proof}

\begin{remark}
	The converse of Corollary \ref{three squares} doesn't hold in
        general. A counterexample is given by $a=163$ and $b=14$, see
        \cite[Theorem 4]{Taussky}. However, it is true that
        $\Q(\sqrt{a})$ is embeddable in a quaternionic extension of
        $\Q$ if and only if $a$ is a sum of three squares (see, e.g.,
        \cite[(II.2.1)]{JensenYui}).
\end{remark}

We now describe the families of quaternionic number fields that we
consider in this section. We shall choose $a$ and $b$ such that $2$ is
unramified in $\Q(\sqrt{a},\sqrt{b})$; that is, we assume that
$a,b\equiv 1 \pmod{4}$, are squarefree and satisfy the assumptions of
Theorem~\ref{quaternionthm}. Consider the family $\FF$ of quaternionic
fields $K_q=\Q(\sqrt{q\theta})$ containing $\Q(\sqrt{a},\sqrt{b})$,
where $a$ and $b$ are fixed, $\theta$ is provided by Theorem
\ref{quaternionthm}, and $q\equiv 0,1 \pmod{4}$ varies over
fundamental discriminants relatively prime to $ab$.

\begin{lemma}\label{l:conductor-quaternion}
The conductor of the Artin representation corresponding to $\Q(\sqrt{q\theta})$ is $2^\alpha r(ab)^2(q^*)^2$, where $\alpha=0,4$ and may depend on $q$, $r(ab)$ is the square-free part of $ab$, and $q^*$ denotes the odd part of $q$.
\end{lemma}
\begin{proof}
This is established by Fr\"ohlich~\cite{Fr} and we outline here a somewhat more direct proof. The splitting pattern of an odd prime dividing $q$ is $(1^2)^4$ or $(2^2)^2$. In both cases the ramification is tame and the inertia subgroup is $\{\pm 1\}$. Since $\chi(-1)=-2$ we find that the local Artin conductor is equal to $2$. The prime $2$ can either be unramified in which case $\alpha=0$, or have the same splitting pattern as above in which case the prime $2$ is wildly ramified and $\alpha=4$. The primes dividing $ab$, which all have tame ramification, can in addition have splitting type $(1^4)^2$ in which case the inertia has order four. Since $\chi(i)=\chi(j)=\chi(k)=0$ the Artin conductor is again equal to $2$. 
\end{proof}

The family $\FF$ can be interpreted as a geometric family in the following way. We consider the binary quadratic form $x^2-q\theta y^2$ over $\Q(\sqrt{a},\sqrt{b})$ and construct  the zero locus $x^2-q\theta y^2=0$ inside $\A^1\times\P^1$. Applying restriction and extension of scalars from $\Q(\sqrt{a},\sqrt{b})$ to $\Q$, we obtain a branched covering $X\to \A^1$ of degree eight defined over $\Q$. Our family $\FF$ is connected to the covering $X\to \A^1$ in the same way as the families and coverings are connected in Section \ref{s:setup}.   

The action of $\Gal(\Q)$ on $K_q=\Q(\sqrt{q\theta})$ splits into the disjoint actions on $\sqrt{q}$ and $\sqrt{\theta}$. It follows that the Artin representations $\rho_{K_q}: \Gal(\Q) \to \GL_2(\C)$ differ by quadratic twists, i.e.\ by twisting by the one dimensional representations $\Gal(\Q)\to \{\pm 1\}$ attached to the quadratic Dirichlet characters $\chi_q$.

We order elements in $\FF$ by conductor, and let $\FF(x)$ denote the set of elements in $\FF$ with conductor less than $x$. Let $\cL=\cL(x)$ be defined by
\begin{equation}\label{eqLq}
\cL:=\frac{1}{|\FF(x)|}\sum_{K\in\FF(x)}\log C_K.
\end{equation}
We write $M=\Q(\sqrt{a},\sqrt{b})$. Before stating our theorem on low-lying zeros, we collect the information on splitting behavior of unramified primes in $K$ and $M$ discussed above, along with the corresponding densities in the family $\FF$, in Table \ref{table2}. In what follows we will find it convenient to denote the splitting type of a prime $p$ in a field $K$ by $\sigma_K(p)$.

\begin{table}[ht]
\centering
\begin{tabular}{|c | c| c| }
  \hline
  &&\\[-8pt]
  Splitting type of $p$ in $K$ & Splitting type of $p$ in $M$ & Density in $\FF$\\[12pt]
  \hline
  &&\\[-8pt]
  $(1)^8$& $(1)^4$& $\displaystyle\frac1{8}+O\Bigl(\frac{1}{p}\Bigr)$\\[12pt]
  $(2)^4$& $(1)^4$& $\displaystyle\frac1{8}+O\Bigl(\frac{1}{p}\Bigr)$\\[12pt]
  $(4)^2$& $(2)^2$& $\displaystyle\frac3{4}+O\Bigl(\frac{1}{p}\Bigr)$\\[12pt]
\hline
\end{tabular}
\caption{Densities of splitting types in $\FF$.}\label{table2}
\end{table}

The Sato-Tate group of $\FF$ is $Q\subset\GL_2(\C)$ embedded via the $2$-dimensional representation $\rho$. As before in \S\ref{s:families-degree-n}
and \S\ref{sub:D4}, we identify $\TT$ with $(S^1)^2$ modulo permutation of the coordinates. Then the Sato-Tate measure equals 
\begin{equation*}
\mu_{\mathrm{ST}}(\FF)=\frac18\delta(1,1)+\frac18\delta(-1,-1)+\frac34\delta(i,-i),
\end{equation*}
and the corresponding indicators are $i_1(\FF)=1$, $i_2(\FF)=1$ and $i_3(\FF)=-1$. In particular it follows that the family $\FF$ is homogeneous symplectic.

\begin{theorem}\label{quaternion-1-level}
Let $\FF$ be the one-parameter family of quaternionic fields described above. If $f$ is a Paley-Wiener function whose Fourier transform has support in $[-\alpha,\alpha]$ for $\alpha<\frac{4}{11}$, then  
\begin{align*}
\lim_{x\to\infty}\frac{1}{|\FF(x)|}\displaystyle\sum_{K\in\FF(x)}\displaystyle\sum_j f\bigg(\frac{\gamma^{(j)}_K\cL}{2\pi}\bigg)=\widehat{f}(0)+\frac{f(0)}{2}.
\end{align*}
\end{theorem}

\begin{proof}
As before, we can without loss of generality assume that $f$ is even. Recalling \eqref{eqexpused}, \eqref{limit1} and the fact that the $L$-functions $L(s,\rho_K)$ are entire, we write
\begin{equation*}%\label{eqquantityq}
	\frac{1}{|\FF(x)|}\displaystyle\sum_{K\in\FF(x)}\displaystyle\sum_j f\bigg(\frac{\gamma^{(j)}_K\cL}{2\pi}\bigg)=\widehat{f}(0)+o(1)-(\CS_1+\CS_2+\CS_3+\CS_\ram),
\end{equation*}
where the $\CS_i$ are defined exactly as in \eqref{eqs}.

To evaluate $\CS_1$, note that the primes $p$ that have splitting type $(2)^2$ in $\Q(\sqrt{a},\sqrt{b})$ do not contribute to the sum over primes because $\theta_K(p)=0$ for all such primes. Thus, we may restrict our sum to primes that split in $\Q(\sqrt{a},\sqrt{b})$, yielding
\begin{equation*}%\label{eqs1q}
\begin{array}{rcl}
\CS_{1}&=&\displaystyle\frac{2}{\cL|\FF(x)|}\displaystyle\sum_{\substack{p\\ \sigma_M(p)=(1)^4}}
\displaystyle\frac{\log p}{\sqrt{p}}\widehat{f}\Bigl(
\displaystyle\frac{\log p}{\cL}\Bigr)
\displaystyle\sum_{K\in\FF^{p\nmid\Delta}(x)}\theta_K(p)
\\[.25in]&=& \displaystyle\frac{2}{\cL|\FF(x)|}\displaystyle\sum_{\substack{p\\ \sigma_M(p)=(1)^4}} \displaystyle\frac{\log p}{\sqrt{p}}\widehat{f}\Bigl(\displaystyle\frac{\log p}{\cL}\Bigr)
\biggl(
\displaystyle\sum_{\substack{K\in\FF^{p\nmid\Delta}(x)\\ \sigma_K(p)=(1)^8}} 2-\displaystyle\sum_{\substack{K\in\FF^{p\nmid\Delta}(x)\\ \sigma_K(p)=(2)^4}}2
\biggr).
\end{array}
\end{equation*}
Let $p$ be a prime that splits in $\Q(\sqrt{a},\sqrt{b})$. We claim that the splitting type of $p$ in the quaternionic field $K_q$ corresponding to the parameter $q$ depends only on whether or not $q$ is a quadratic residue modulo $p$. Indeed, by assumption, we have $\Q(\sqrt{a},\sqrt{b})\otimes\Q_p=\Q_p^4$. So the splitting type of $p$ in $K_q=\Q(\sqrt{q\theta })$ depends only on whether or not $q\theta$ is a square in $\Q_p$. The claim now follows since $\theta$ is fixed.
Therefore, using Burgess's bound as in \cite[Lemma 2.3]{Munsch}, we have
\begin{equation*}%\label{eqs1q}
\begin{array}{rcl}
\#\{K\in\FF^{p\nmid\Delta}(x):\sigma_K(p)=(1)^8\} &=& \displaystyle\frac{p-1}{2p}|\FF^{p\nmid\Delta}(x)|+O_{\epsilon}\big(x^{1/4}(\log x)p^{3/16+\epsilon}\big),\\[.1in]
\#\{K\in\FF^{p\nmid\Delta}(x):\sigma_K(p)=(2)^4\} &=& \displaystyle\frac{p-1}{2p}|\FF^{p\nmid\Delta}(x)|+O_{\epsilon}\big(x^{1/4}(\log x)p^{3/16+\epsilon}\big).
\end{array}
\end{equation*}
Hence
$$
\CS_1=O_{\epsilon}\biggl(\frac{1}{x^{1/4}}\sum_{\log p\leq \cL\alpha}\frac{\log p}{p^{5/16-\epsilon}}\biggr)=O_{\epsilon}\biggl(\frac{e^{(11/16+\epsilon)\cL\alpha}}{x^{1/4}}\biggr).
$$

To estimate $\CS_2$, we note that $\theta_K(p^2)$ is $2$ or $-2$ depending on whether the splitting type of $p$ in $\Q(\sqrt{a},\sqrt{b})$ is $(1)^4$ or $(2)^2$, respectively. Therefore, we have
\begin{equation*}%\label{eqs2q}
\begin{array}{rcl}
\CS_{2}&=&\displaystyle\frac{2}{\cL|\FF(x)|}\Bigg(\displaystyle\sum_{\substack{p\\ \sigma_M(p)=(1)^4}}
\displaystyle\frac{\log p}{p}\widehat{f}\Bigl(
\displaystyle\frac{2\log p}{\cL}\Bigr)
\displaystyle\sum_{K\in\FF^{p\nmid\Delta}(x)}2-\displaystyle\sum_{\substack{p\\ \sigma_M(p)=(2)^2}}
\displaystyle\frac{\log p}{p}\widehat{f}\Bigl(
\displaystyle\frac{2\log p}{\cL}\Bigr)
\displaystyle\sum_{K\in\FF^{p\nmid\Delta}(x)}2\Bigg)\\[.3in]
&=&\displaystyle\frac{4}{\cL|\FF(x)|}\Bigg(\displaystyle\sum_{\substack{p\\ \sigma_M(p)=(1)^4}}
\displaystyle\frac{\log p}{p}\widehat{f}\Bigl(\displaystyle\frac{2\log p}{\cL}\Bigr)
|\FF(x)|\bigl(1+O(p^{-1})\bigr)
-\displaystyle\sum_{\substack{p\\ \sigma_M(p)=(2)^2}}
\displaystyle\frac{\log p}{p}\widehat{f}\Bigl(\displaystyle\frac{2\log p}{\cL}\Bigr)
|\FF(x)|\bigl(1+O(p^{-1})\bigr)
\Bigg)\\[.3in]
&=&\displaystyle\frac{4}{\cL}\Bigg(
\displaystyle\sum_{\substack{p\\ \sigma_M(p)=(1)^4}}
\displaystyle\frac{\log p}{p}\widehat{f}\Bigl(\displaystyle\frac{2\log p}{\cL}\Bigr)
-\displaystyle\sum_{\substack{p\\ \sigma_M(p)=(2)^2}}
\displaystyle\frac{\log p}{p}\widehat{f}\Bigl(\displaystyle\frac{2\log p}{\cL}\Bigr)
\Bigg)+o(1).
\end{array}
\end{equation*}
Hence, recalling that the Frobenius-Schur indicator of $\rho$ equals
$-1$ along with the relation \eqref{eqtempeval1}, we obtain
\begin{equation*}
\CS_{2}=-\frac{f(0)}{2}+o(1).
\end{equation*}

Finally, bounding the quantities $\CS_3$ and $\CS_\ram$ in exactly the same way as in \eqref{eqs23ram}, we find that
\begin{equation*}%\label{eqs34q}
\begin{array}{rcl}
\CS_3&=& o(1),\\[.1in]
\CS_\ram&=&o(1).
\end{array}
\end{equation*}
We conclude that
\begin{equation*}
\frac{1}{|\FF(x)|}\displaystyle\sum_{K\in\FF(x)}\displaystyle\sum_j f\bigg(\frac{\gamma^{(j)}_K\cL}{2\pi}\bigg)=\widehat{f}(0)+\frac{f(0)}{2}+O_{\epsilon}\biggl(\frac{e^{(11/16+\epsilon)\cL\alpha}}{x^{1/4}}\biggr)+o(1),
\end{equation*}
from which the desired result readily follows.
\end{proof}

Let us remark that the condition $\alpha<\frac4{11}$ on the support in Theorem \ref{quaternion-1-level} can be relaxed. In fact, it follows from a result of Rubinstein \cite{Rubinstein} on more general families of quadratic twists that Theorem \ref{quaternion-1-level} holds with any $\alpha<\frac12$. Rubinstein uses a large sieve type inequality due to Jutila instead of Burgess's bound to obtain this result. Furthermore, Katz and Sarnak \cite[Appendix 1 (unpublished)]{KatzSarnak2} proved, assuming GRH, that the above result holds also in the doubled region $\alpha<1$. For any integer $n\geq 1$, \cite{Rubinstein} further establishes the $n$-level density for test functions with support restricted to the region $\sum_{i=1}^n|x_i|<\frac1m$ of low-lying zeros of families consisting of quadratic twists of a fixed automorphic cuspidal representation of $\GL_m(\Q)$. It would be interesting to see if, similarly as done in \cite{ERGR} in the case $m=1$, one could double the support (conditional on GRH) in the $n$-level result for the family $\FF$ considered in this section. 

\begin{proposition}\label{p:root}
The root number of the Artin representation attached to $\Q(\sqrt{q\theta})$ is independent of $q$.
\end{proposition}

\begin{proof}
See Fr\"ohlich \cite[Assertion XV]{Fr}. We give here a proof based on the properties of local epsilon factors in~\cite{Tate}. 
Let $\rho$ be the Artin representation attached to $\Q(\sqrt{\theta})$ and $\chi_q$ be the quadratic Dirichlet character attached to $q$. We want to prove that $\epsilon(\rho\otimes \chi_q)=\epsilon(\rho)$.

Fix the standard additive character $\psi$ of $\Q\backslash \mathbb{A}$. Then the epsilon factor splits as a product of local epsilon factors $\epsilon_p(\rho\otimes \chi_q,\psi)$. We shall verify that for every $p$, 
\[
\epsilon_p(\rho\otimes \chi_q,\psi) = \epsilon_p(\rho,\psi) \epsilon_p(\chi_q,\psi)^2.
\]
Since globally $\epsilon(\chi_q)=1$ this will finish the proof.

For $p=\infty$ this can be verified directly. For $p\nmid 2abq$ each epsilon factor is equal to one.  For $p\mid ab$, we have $p\nmid q$ by assumption.  Thus $\epsilon_p(\rho\otimes \chi_q,\psi) = \epsilon_p(\rho,\psi) \chi_q(p^{v_p(r(ab)^2)}) = \epsilon_p(\rho,\psi)$, where we have used (Lemma~\ref{l:conductor-quaternion}) that the conductor of $\rho$ is equal to $2^{\alpha}r(ab)^2$ which is a perfect square.  For $p\mid 2q$, we have $p \nmid ab$ by assumption. Since $\det \rho$ is trivial because $\rho$ has image in $\SL_2(\C)$, we find that  $\epsilon_p(\rho\otimes \chi_q,\psi)= \epsilon_p(\chi_q,\psi)^2$. This concludes the claim.
\end{proof}

\AtEndDocument{\bigskip{\footnotesize%
  \textit{E-mail address}, \texttt{arul.shnkr@gmail.com}\par
  \textsc{Department of Mathematics, Harvard University, Cambridge, MA 02138, USA}\par
  \addvspace{\medskipamount}
  \textit{E-mail address}, \texttt{c.a.sodergren@gmail.com}\par
  \textsc{Department of Mathematical Sciences, University of Copenhagen, Universitetsparken 5, \\
  \rule[0ex]{0ex}{0ex}\hspace{36pt}2100 Copenhagen, Denmark}\par
  \addvspace{\medskipamount}
  \textit{E-mail address}, \texttt{templier@math.cornell.edu}\par
  \textsc{Department of Mathematics, Cornell University, Ithaca, NY 14853, USA}
}}

\end{document}